\documentclass[12pt]{article}
\usepackage{amsmath, amsthm, amsfonts, amssymb}
\usepackage{mathrsfs}
\usepackage{bbm}
\usepackage{amssymb}
\usepackage{txfonts}
\usepackage{graphicx}
\usepackage{wrapfig}

\newlength{\defbaselineskip}
\newcounter{marnote}

\newcommand{\setlinespacing}[1]%
           {\setlength{\baselineskip}{#1 \defbaselineskip}}

\theoremstyle{plain}
\newtheorem{theorem}{Theorem}[section]

\newtheorem{lemma}[theorem]{Lemma}
\newtheorem{prop}[theorem]{Proposition}
\theoremstyle{definition}

\theoremstyle{remark}
\newtheorem{remark}{Remark}[section]
\numberwithin{equation}{section}

\begin{document}


\title{Gradient estimates for solutions of the Lam\'{e} system with partially infinite coefficients in  dimensions greater than two}

\author{JiGuang Bao\footnote{School of Mathematical Sciences, Beijing Normal University, Laboratory of Mathematics and Complex Systems, Ministry of
   Education, Beijing 100875,
China. Email: jgbao@bnu.edu.cn.}\quad HaiGang
Li\footnotemark[1]~\footnote{Corresponding author. Email:
hgli@bnu.edu.cn.}\quad and\quad YanYan Li\footnote{Department of
Mathematics, Rutgers University, 110 Frelinghuysen Rd, Piscataway,
NJ 08854, USA. Email: yyli@math.rutgers.edu.} }

\maketitle


\begin{abstract}
We establish upper bounds on the blow-up rate of the gradients of solutions of the Lam\'{e} system with partially infinite coefficients in dimensions greater than two as the distance between the surfaces of discontinuity of the coefficients of the system tends to zero.
\end{abstract}

\section{Introduction and main results}

In this paper, we establish upper bounds on the blow-up rate of the gradients of solutions of the Lam\'{e} system with partially infinite coefficients in dimensions greater than two as the distance between the surfaces of discontinuity of the coefficients of the system tends to zero. This work is stimulated by the study of Babu\u{s}ka, Andersson, Smith and Levin in \cite{ba} concerning initiation and growth of damage in composite materials.  The Lam\'{e} system is assumed and they computationally analyzed the damage and fracture in composite materials. They observed numerically that the size of the strain tensor remains bounded when the distance $\epsilon$, between two inclusions, tends to zero.
This was  proved by Li and Nirenberg in \cite{ln}.  Indeed such $\epsilon$-independent gradient estimates was established there for solutions of divergence form second order elliptic systems, including linear systems of elasticity, with piecewise H\"{o}lder continuous coefficients in all dimensions. See Bonnetier and Vogelius \cite{bv}
and Li and Vogelius  \cite{lv}  for corresponding results on divergence form elliptic equations.

The estimates in \cite{ln} and \cite{lv} depend on the ellipticity of the coefficients. If ellipticity constants are allowed to deteriorate, the situation is very different.
Consider the scalar equation
\begin{equation}\label{equk}
\begin{cases}
\nabla\cdot\Big(a_{k}(x)\nabla{u}_{k}\Big)=0&\mbox{in}~\Omega,\\
u_{k}=\varphi&\mbox{on}~\partial\Omega,
\end{cases}
\end{equation}
where $\Omega$ is a bounded open set of $\mathbb{R}^{d}$, $d\geq2$, containing two $\epsilon$-apart convex inclusions $D_{1}$ and $D_{2}$, $\varphi\in{C}^{2}(\partial\Omega)$ is given, and
$$a_{k}(x)=
\begin{cases}
k\in(0,\infty)&\mbox{in}~D_{1}\cup{D}_{2},\\
1&\mbox{in}~\Omega\setminus\overline{D_{1}\cup{D}_{2}}.
\end{cases}
$$
When $k=\infty$,
the $L^\infty$-norm of $|\nabla u_\infty|$ for the solutions $u_{\infty}$ of  (\ref{equk}) generally becomes unbounded as $\epsilon$ tends to $0$.  The blow up rate of $|\nabla u_\infty|$ is respectively
 $\epsilon^{-1/2}$ in dimension $d=2$,
$(\epsilon|\ln\epsilon|)^{-1}$ in
dimension $d=3$,  and  $\epsilon^{-1}$  in dimension $d\ge 4$. See
Bao, Li and Yin \cite{bly1}, as well as Budiansky and Carrier \cite{bc},
  Markenscoff \cite{m}, Ammari, Kang and Lim \cite{akl},
 Ammari, Kang, Lee, Lee and Lim \cite{aklll} and Yun  \cite{y1,y2}.
Further, more detailed, characterizations of the singular behavior of $\nabla{u}_{\infty}$ have been obtained by Ammari, Ciraolo, Kang, Lee and Yun \cite{ackly}, Ammari, Kang, Lee, Lim and Zribi \cite{AKLLZ}, Bonnetier and Triki \cite{bt0, bt}, Gorb and Novikov \cite{gn} and
 Kang, Lim and Yun \cite{kly, kly2}. For related works, see \cite{abtv, adkl, agkl, bly2,  bt, bcn, dong, dongxiong, dongzhang, FKNN, keller1, keller2, Li-Li, llby, ly, ly2, MNT, x, BaoXiong} and the references therein.

In this paper, we mainly investigate the gradient estimates for the Lam\'{e} system with partially infinite coefficients in dimension $d=3$, a physically relevant dimension.
This paper is a continuation of \cite{bll}, where the estimate for dimension $d=2$, another physically relevant dimension, is established. We prove that $(\epsilon|\ln\epsilon|)^{-1}$ is an   upper bound of the blow up rate of the strain tensor in dimension three, the same as the scalar equation case mentioned above. New difficulties need to be overcome, and a number of refined estimates, via appropriate iterations, are used in our proof. We also prove that $\epsilon^{-1}$ is an upper bound of the  blow up rate of the strain tensor in dimension $d\ge 4$, which is also the same as the scalar equation case. Note that it has been proved in \cite{bly1} that these upper bounds in dimension $d\ge 3$ are optimal in the scalar equation case.

We consider the Lam\'{e} system in linear elasticity with piecewise constant coefficients, which is stimulated by the study of composite media with closely spaced interfacial boundaries. Let $\Omega\subset\mathbb{R}^{3}$ be a bounded open set with $C^{2}$ boundary, and $D_{1}$ and $D_{2}$ are two disjoint convex open sets in $\Omega$ with
$C^{2,\gamma}$ boundaries, $0<\gamma<1$, which are $\epsilon$ apart and far away from $\partial{\Omega}$, that is,
\begin{equation}\label{omega}
\begin{array}{l}
\displaystyle \overline{D}_{1},\overline{D}_{2}\subset\Omega,\quad\mbox{the principle curvatures of }\partial{D}_{1},\partial{D}_{2}\geq\kappa_{0}>0,\\
\displaystyle \epsilon:=\mathrm{dist}(D_{1},D_{2})>0,\quad \mathrm{dist}(D_{1}\cup{D}_{2},\partial{\Omega})>\kappa_{1}>0,
\end{array}
\end{equation}
where $\kappa_{0},\kappa_{1}$ are constants independent of $\epsilon$. We also assume that the $C^{2,\gamma}$ norms of $\partial{D}_{i}$ are bounded by some constant independent of $\epsilon$. This implies that each $D_{i}$ contains a ball of radius $r_{0}^{*}$ for some constant $r_{0}^{*}>0$ independent of $\epsilon$. Denote $$\widetilde{\Omega}:=\Omega\setminus\overline{D_{1}\cup{D}_{2}}.$$
Assume that $\widetilde{\Omega}$ and $D_{1}\cup{D}_{2}$ are occupied, respectively, by two different isotropic and homogeneous materials with different Lam\'{e} constants $(\lambda,\mu)$ and $(\lambda_{1},\mu_{1})$. Then the elasticity tensors for the inclusions and the background can be written, respectively, as  $\mathbb{C}^{1}$ and $\mathbb{C}^{0}$, with
$$C_{ij\,kl}^{1}=\lambda_{1}\delta_{ij}\delta_{kl}+\mu_{1}(\delta_{ik}\delta_{jl}+\delta_{il}\delta_{jk}),$$
and
\begin{equation}\label{C0}
C_{ij\,kl}^{0}=\lambda\delta_{ij}\delta_{kl}+\mu(\delta_{ik}\delta_{jl}+\delta_{il}\delta_{jk}),
\end{equation}
where $i,j,k,l=1,2,3$ and $\delta_{ij}$ is the Kronecker symbol: $\delta_{ij}=0$ for $i\neq{j}$, $\delta_{ij}=1$ for $i=j$.
Let $u=\left(u_{1},u_{2},u_{3}\right)^{T}:~\Omega\rightarrow\mathbb{R}^{3}$ denote the displacement field.
For a given vector valued function $\varphi$, we consider the following Dirichlet problem for the Lam\'{e} system
\begin{equation}\label{system}
\begin{cases}
\nabla\cdot\bigg(\left(\chi_{\widetilde{\Omega}}\mathbb{C}^{0}+\chi_{D_{1}\cup{D}_{2}}\mathbb{C}^{1}\right)e(u)\bigg)=0,
&\mbox{in}~\Omega,\\
u=\mathbf{\varphi},&\mbox{on}~\partial{\Omega},
\end{cases}
\end{equation}
where $\chi_{D}$ is the characteristic function of $D\subset\mathbb{R}^{3}$,
$$e(u)=\frac{1}{2}\left(\nabla{u}+(\nabla{u})^{T}\right)$$
is the strain tensor.

Assume that the standard ellipticity condition holds for \eqref{system}, that is,
\begin{equation}\label{coeff3_convex}
\mu>0,\quad\,3\lambda+2\mu>0;\quad\quad\mu_{1}>0,\quad\,3\lambda_{1}+2\mu_{1}>0.
\end{equation}
For $\varphi\in{H}^{1}(\Omega;\mathbb{R}^{3})$, it is well known that there exists a unique solution $u\in{H}^{1}(\Omega;\mathbb{R}^{3})$ of the Dirichlet problem \eqref{system}, which is also the minimizer of the energy functional
$$J_{1}[u]=\frac{1}{2}\int_{\Omega}
\bigg(\left(\chi_{\widetilde{\Omega}}\mathbb{C}^{0}+\chi_{D_{1}\cup{D}_{2}}\mathbb{C}^{1}\right)e(u),e(u)\bigg)dx$$
on
$${H}_{\varphi}^{1}(\Omega;\mathbb{R}^{3}):=\left\{~u\in{H}^{1}(\Omega;\mathbb{R}^{3})~\big|~u-\varphi\in{H}_{0}^{1}(\Omega;\mathbb{R}^{3})~\right\}.$$
More details can be found in the Appendix in \cite{bll}.

Introduce the linear space of rigid displacement in $\mathbb{R}^{3}$,
$$\Psi:=\bigg\{\psi\in{C}^{1}(\mathbb{R}^{3};\mathbb{R}^{3})~\big|~\nabla\psi+(\nabla\psi)^{T}=0~\bigg\},$$
equivalently,
$$\Psi=\mathrm{span}
\left\{~\psi^{1}=
\begin{pmatrix}
~ 1~ \\
 0 \\
 0 \\
\end{pmatrix},~\psi^{2}=\begin{pmatrix}
 ~0~ \\
 1 \\
 0 \\
\end{pmatrix},~\psi^{3}=\begin{pmatrix}
~ 0~ \\
 0 \\
 1 \\
\end{pmatrix},~\psi^{4}=\begin{pmatrix}
x_{2} \\
-x_{1} \\
0 \\
\end{pmatrix},~\psi^{5}=\begin{pmatrix}
x_{3} \\
0 \\
-x_{1} \\
\end{pmatrix},~\psi^{6}=\begin{pmatrix}
0 \\
x_{3} \\
-x_{2} \\
\end{pmatrix}~\right\}.
$$
If $\xi\in{H}^{1}(D; \mathbb{R}^{3})$, $e(\xi)=0$ in $D$,
and $D\subset \mathbb{R}^{3}$ is a connected  open set, then $\xi$ is a linear combination of $\{\psi^{\alpha}\}$ in $D$. If an element $\xi$ in $\Psi$ vanishes at three non-collinear points, then $\xi\equiv0$, see Lemma \ref{lem_xi}.

For fixed $\lambda$ and $\mu$ satisfying $\mu>0$ and $3\lambda+2\mu>0$, denoting $u_{\lambda_{1},\mu_{1}}$ the solution of \eqref{system}. Then, as proved in the Appendix in \cite{bll},
$$u_{\lambda_{1},\mu_{1}}\rightarrow\,u\quad\mbox{in}~H^{1}(\Omega;\mathbb{R}^{3})\quad
\mbox{as}~~\min\{\mu_{1},3\lambda_{1}+2\mu_{1}\}\rightarrow\infty,$$
where $u$ is a $H^{1}(\Omega;\mathbb{R}^{3})$ solution of
\begin{equation}\label{mainequation}
\begin{cases}
\mathcal{L}_{\lambda,\mu}u:=\nabla\cdot\left(\mathbb{C}^{0}e(u)\right)=0,&\mbox{in}~\widetilde{\Omega},\\
u\big|_{+}=u\big|_{-},&\mbox{on}~\partial{D}_{1}\cup\partial{D}_{2},\\
e(u)=0,&\mbox{in}~D_{1}\cup{D}_{2},\\
\int_{\partial{D}_{i}}\frac{\partial{u}}{\partial\nu_{0}}\big|_{+}\cdot\psi^{\alpha}=0,&i=1,2,~~\alpha=1,2,\cdots,6,\\
u=\varphi,&\mbox{on}~\partial{\Omega},
\end{cases}
\end{equation}
where
$$\frac{\partial{u}}{\partial\nu_{0}}\bigg|_{+}:=\left(\mathbb{C}^{0}e(u)\right)\vec{n}
=\lambda\left(\nabla\cdot{u}\right)\vec{n}+\mu\left(\nabla{u}+(\nabla{u})^{T}\right)\vec{n}.$$
and $\vec{n}$ is the unit outer normal of $D_{i}$, $i=1,2$. Here and throughout this paper the subscript $\pm$ indicates the limit from outside and inside the domain, respectively. In this paper we study solutions of \eqref{mainequation}, a Lam\'{e} system with infinite coefficients in $D_{1}\cup{D}_{2}$.

The existence, uniqueness and regularity of weak solutions of \eqref{mainequation}, as well as a variational formulation, can be found in the Appendix in \cite{bll}. In particular, the  $H^{1}$ weak solution is in ${C}^{1}(\overline{\widetilde{\Omega}};\mathbb{R}^{3})\cap{C}^{1}(\overline{D_{1}\cup{D}_{2}};\mathbb{R}^{3})$. The solution is also the unique
function which has the least energy in appropriate functional spaces,
characterized by
$$
I_{\infty}[u]=\min_{v\in\mathcal{A}}I_{\infty}[v],
$$
where
$$
I_{\infty}[v]:=
\frac{1}{2}\int_{\widetilde \Omega}\left(\mathbb{C}^{(0)}
e(v),e(v)\right)dx,
$$
and
$$
\mathcal{A}:=\left\{u\in{H}_{\varphi}^{1}(\Omega;\mathbb{R}^{3})
~\big|~e(u)=0~~\mbox{in}~D_1\cup D_2\right\}.
$$

It is well known, see \cite{osy},  that for any open set $O$ and $u,v\in{C}^{2}(O)$,
\begin{equation}\label{equ4-1}
\int_{O}\left(\mathbb{C}^{0}e(u),e(v)\right)dx
=-\int_{O}\left(\mathcal{L}_{\lambda,\mu}u\right)\cdot{v}
+\int_{\partial O}\frac{\partial{u}}{\partial\nu_{0}}\bigg|_{+}\cdot{v}.
\end{equation}
A calculation gives
\begin{align}\label{L_u}
\left(\mathcal{L}_{\lambda,\mu}u\right)_{k}=\mu\Delta{u}_{k}+(\lambda+\mu)\partial_{x_{k}}\left(\nabla\cdot{u}\right),\quad\,k=1,2,3.
\end{align}
We assume that for some $\delta_{0}>0$,
\begin{equation}\label{coeff4_strongelyconvex}
\delta_{0}\leq\mu,3\lambda+2\mu\leq\frac{1}{\delta_{0}}.
\end{equation}
Since $D_{1}$ and $D_{2}$ are two strictly convex subdomains of $\Omega$, there exist two points $P_{1}\in\partial{D}_{1}$ and $P_{2}\in\partial{D}_{2}$ such that
\begin{equation}\label{P1P2}
\mathrm{dist}(P_{1},P_{2})=\mathrm{dist}(\partial{D}_{1},\partial{D}_{2})=\epsilon.
\end{equation}
Use $\overline{P_{1}P_{2}}$ to denote the line segment connecting $P_{1}$ and $P_{2}$. 
Throughout the paper, unless otherwise stated, $C$ denotes a constant, whose values may vary from line to line, depending only on $d,\kappa_{0},\kappa_{1},\gamma,\delta_{0}$,  and an upper bound of the $C^{2}$ norm of $\partial\Omega$ and the $C^{2,\gamma}$ norms of $\partial{D}_{1}$ and $\partial{D}_{2}$, but not on $\epsilon$. Also, we call a constant having such dependence a {\it universal constant}. The main result of this paper is for dimension three.

\begin{theorem}\label{mainthm1}
Assume that $\Omega$, $D_{1},D_{2}$, $\epsilon$ are defined in \eqref{omega}, $\lambda$ and $\mu$ satisfy \eqref{coeff4_strongelyconvex} for some $\delta_{0}>0$,  and $\varphi\in{C}^{2}(\partial\Omega;\mathbb{R}^{3})$. Let $u\in{H}^{1}(\Omega;\mathbb{R}^{3})\cap{C}^{1}(\overline{\widetilde{\Omega}};\mathbb{R}^{3})$ be the solution of \eqref{mainequation}. Then for $0<\epsilon<1/2$, we have
\begin{equation}\label{normbound}
\|\nabla{u}\|_{L^{\infty}(\Omega;\mathbb{R}^{3})}
\leq\frac{C}{\epsilon|\ln\epsilon|}\|\varphi\|_{C^{2}(\partial{\Omega};\mathbb{R}^{3})},
\end{equation}
where $C$ is a universal constant.
\end{theorem}

\begin{remark}\label{rem1}
The proof of Theorem \ref{mainthm1} actually gives the following stronger estimates:
\begin{equation}\label{mainestimates}
|\nabla{u}(x)|\leq\bigg(\dfrac{C}{|\ln\epsilon|\big(\epsilon+\mathrm{dist}^{2}(x,\overline{P_{1}P_{2}})\big)}
+\dfrac{C\mathrm{dist}(x,\overline{P_{1}P_{2}})}{\epsilon+\mathrm{dist}^{2}(x,\overline{P_{1}P_{2}})}\bigg)
~\|\varphi\|_{C^{2}(\partial{\Omega};\mathbb{R}^{3})},
~~x\in\widetilde{\Omega},
\end{equation}
and
\begin{equation}\label{mainestimatesD1D2}
|\nabla{u}(x)|\leq\,C\|\varphi\|_{C^{2}(\partial{\Omega};\mathbb{R}^{3})},\qquad\,x\in{D}_{1}\cup{D}_{2}.
\end{equation}
\end{remark}

\begin{remark}\label{rem2}
The strict convexity assumption on $\partial D_1$ and $\partial D_2$
can be replaced by a weaker relative strict convexity assumption, see \eqref{h1h23} in Section \ref{sec_gradient123}.
\end{remark}

\begin{remark}\label{rem3}
Here $\varphi\in{C}^{2}(\partial\Omega;\mathbb{R}^{3})$ can be replaced by $\varphi\in{H}^{1/2}(\partial\Omega;\mathbb{R}^{3})$. Indeed, the $H^{1}$ norm of the solution $u$ in $\Omega$ is bounded by a universal constant. Then standard elliptic estimates give a universal bound of $u$ in $C^{2}$ norm in $\left\{x\in\Omega~\big|~\frac{\kappa_{1}}{4}<\mathrm{dist}(x,\partial\Omega)<\frac{\kappa_{1}}{2}\right\}$. We apply the theorem in $\Omega':=\left\{x\in\Omega~\big|~\mathrm{dist}(x,\partial\Omega)>\frac{\kappa_{1}}{3}\right\}$ with $\varphi':=u\big|_{\partial\Omega'}$.
\end{remark}

\begin{remark}\label{rem4}
Since the blow up rate of $|\nabla u_\infty|$ for solutions
of the scalar equation (\ref{equk}) when $k=\infty$ is known to reach the magnitude
 $(\epsilon|\ln\epsilon|)^{-1}$ in dimension three, see \cite{bly1},
estimate (\ref{normbound}) is expected to be
optimal.
\end{remark}

Following arguments in the proof of Theorem \ref{mainthm1}, we establish the corresponding estimates for higher dimensions $d\geq4$. Let $\Omega\subset\mathbb{R}^{d}$, $d\geq4$ be a bounded open set with $C^{2}$ boundary, and $D_{1}$ and $D_{2}$ are two disjoint convex open sets in $\Omega$ with
$C^{2,\gamma}$ boundaries, satisfying \eqref{omega}.
Let $\mathbb{C}^{0}$ be given by \eqref{C0} with $i,j,k,l=1,2,\cdots,d$, where $\lambda$ and $\mu$ satisfy
$$\mu>0,\quad\,d\lambda+2\mu>0,$$
and
\begin{equation}\label{def_Psid}
\Psi:=\bigg\{\psi\in{C}^{1}(\mathbb{R}^{d};\mathbb{R}^{d})~\big|~\nabla\psi+(\nabla\psi)^{T}=0~\bigg\}
\end{equation}
be the linear space of rigid displacement in $\mathbb{R}^{d}$. With $e_{1},\cdots,e_{d}$ denoting the standard basis of $\mathbb{R}^{d}$, $$\left\{~e_{i},~x_{j}e_{k}-x_{k}e_{j}~\big|~1\leq\,i\leq\,d,~1\leq\,j<k\leq\,d~\right\}$$ is a basis of $\Psi$. Denote the basis of $\Psi$ as $\{\psi^{\alpha}\}$, $\alpha=1,2,\cdots,\frac{d(d+1)}{2}$. Consider
\begin{equation}\label{mainequationd4}
\begin{cases}
\mathcal{L}_{\lambda,\mu}u:=\nabla\cdot\left(\mathbb{C}^{0}e(u)\right)=0,&\mbox{in}~\widetilde{\Omega},\\
u\big|_{+}=u\big|_{-},&\mbox{on}~\partial{D}_{1}\cup\partial{D}_{2},\\
e(u)=0,&\mbox{in}~D_{1}\cup{D}_{2},\\
\int_{\partial{D}_{i}}\frac{\partial{u}}{\partial\nu_{0}}\big|_{+}\cdot\psi^{\alpha}=0,&i=1,2,~~\alpha=1,2,\cdots,\frac{d(d+1)}{2},\\
u=\varphi,&\mbox{on}~\partial{\Omega}.
\end{cases}
\end{equation}
Then we have

\begin{theorem}\label{mainthm2}
Assume as above, and $\varphi\in{C}^{2}(\partial\Omega;\mathbb{R}^{d})$, $d\geq4$. Let $u\in{H}^{1}(\Omega;\mathbb{R}^{d})\cap{C}^{1}(\overline{\widetilde{\Omega}};\mathbb{R}^{d})$ be the solution of \eqref{mainequationd4}. Then for $0<\epsilon<1/2$, we have
\begin{equation}\label{normbound2}
\|\nabla{u}\|_{L^{\infty}(\Omega;\mathbb{R}^{d})}
\leq\frac{C}{\epsilon}\|\varphi\|_{C^{2}(\partial{\Omega};\mathbb{R}^{d})},
\end{equation}
where $C$ is a universal constant.
\end{theorem}

\begin{remark}\label{rem5}
The proof of Theorem \ref{mainthm2} actually gives the following stronger estimate in dimension $d\geq4$:
\begin{equation*}\label{mainestimates4}
|\nabla{u}(x)|\leq
\begin{cases}
\dfrac{C}{\epsilon+\mathrm{dist}^{2}(x,\overline{P_{1}P_{2}})}
\|\varphi\|_{C^{2}(\partial{\Omega};\mathbb{R}^{d})},
&\,x\in\widetilde{\Omega},\\\\
C\|\varphi\|_{C^{2}(\partial{\Omega};\mathbb{R}^{d})},&\,x\in{D}_{1}\cup{D}_{2}.
\end{cases}
\end{equation*}
We also have Remarks \ref{rem2}--\ref{rem4} accordingly.
\end{remark}

The rest of this paper is organized as follows.
In Section \ref{sec_outline_of_mainthm1}, we first introduce a setup for the
proof of Theorem \ref{mainthm1}. Then
we  state a proposition,
Proposition \ref{prop_gradient},
 containing  key estimates, and
deduce  Theorem \ref{mainthm1} from the proposition.
In Sections \ref{sec_gradient123} and \ref{sec_C1C2}, we prove Proposition \ref{prop_gradient}. The proof of Theorem \ref{mainthm2} is given in Section \ref{sec_higherdimensions}.
A linear algebra lemma, Lemma \ref{lem_matrix}, used in the proof of Theorem \ref{mainthm1}, is given in Section \ref{sec_appendix}.

\section{Outline of the Proof of Theorem \ref{mainthm1}}\label{sec_outline_of_mainthm1}

The proof of Theorem \ref{mainthm1} makes use of the following decomposition. By the third line of \eqref{mainequation}, $u$ is a linear combination of $\{\psi^{\alpha}\}$ in $D_{1}$ and $D_{2}$, respectively. Since it is clear that $\mathcal{L}_{\lambda,\mu}\xi=0$ in $\widetilde{\Omega}$ and $\xi=0$ on $\partial\widetilde{\Omega}$ imply that $\xi=0$ in $\widetilde{\Omega}$, we decompose the solution of \eqref{mainequation}, as in \cite{bll},  as follows:
\begin{equation}\label{decom_u}
u=
\begin{cases}
\sum\limits_{\alpha=1}^{6}C_{1}^{\alpha}\psi^{\alpha},&\mbox{in}~\overline{D}_{1},\\
\sum\limits_{\alpha=1}^{6}C_{2}^{\alpha}\psi^{\alpha},&\mbox{in}~\overline{D}_{2},\\
\sum\limits_{\alpha=1}^{6}C_{1}^{\alpha}v_{1}^{\alpha}+\sum\limits_{\alpha=1}^{6}C_{2}^{\alpha}v_{2}^{\alpha}+v_{0},&\mbox{in}~\widetilde{\Omega},
\end{cases}
\end{equation}
where $v_{i}^{\alpha}\in{C}^{1}(\overline{\widetilde{\Omega}};\mathbb{R}^{3})$, $i=1,2$, $\alpha=1,2,\cdots,6$,  and $v_{0}\in{C}^{1}(\overline{\widetilde{\Omega}};\mathbb{R}^{3})$ are respectively the solution of
\begin{equation}\label{v1alpha}
\begin{cases}
\mathcal{L}_{\lambda,\mu}v_{i}^{\alpha}=0,&\mbox{in}~\widetilde{\Omega},\\
v_{i}^{\alpha}=\psi^{\alpha},&\mbox{on}~\partial{D}_{i},\\
v_{i}^{\alpha}=0,&\mbox{on}~\partial{D}_{j}\cup\partial{\Omega},~j\neq\,i,
\end{cases}
\end{equation}
and
\begin{equation}\label{v3}
\begin{cases}
\mathcal{L}_{\lambda,\mu}v_{0}=0,&\mbox{in}~\widetilde{\Omega},\\
v_{0}=0,&\mbox{on}~\partial{D}_{1}\cup\partial{D}_{2},\\
v_{0}=\varphi,&\mbox{on}~\partial{\Omega}.
\end{cases}
\end{equation}
The constants $C_{i}^{\alpha}:=C_{i}^{\alpha}(\epsilon)$, $i=1,2$, $\alpha=1,2,\cdots,6$, are uniquely determined by $u$.

By the decomposition \eqref{decom_u}, we write
\begin{equation}\label{nablau_dec}
\nabla{u}=\sum_{\alpha=1}^{3}\left(C_{1}^{\alpha}-C_{2}^{\alpha}\right)\nabla{v}_{1}^{\alpha}
+\sum_{\alpha=1}^{3}C_{2}^{\alpha}\nabla(v_{1}^{\alpha}+v_{2}^{\alpha})
+\sum_{i=1}^{2}\sum_{\alpha=4}^{6}C_{i}^{\alpha}\nabla{v}_{i}^{\alpha}+\nabla{v}_{0},\quad\mbox{in}~\widetilde{\Omega},
\end{equation}
then
\begin{equation}\label{nablau}
\left|\nabla{u}\right|
\leq\sum_{\alpha=1}^{3}\left|C_{1}^{\alpha}-C_{2}^{\alpha}\right|\left|\nabla{v}_{1}^{\alpha}\right|
+\sum_{\alpha=1}^{3}\left|C_{2}^{\alpha}\right|\left|\nabla(v_{1}^{\alpha}+v_{2}^{\alpha})\right|
+\sum_{i=1}^{2}\sum_{\alpha=4}^{6}\left|C_{i}^{\alpha}\right|\left|\nabla{v}_{i}^{\alpha}\right|+\big|\nabla{v}_{0}\big|,\quad\mbox{in}~\widetilde{\Omega}.
\end{equation}
The proof of Theorem \ref{mainthm1} can be reduced to the following proposition. Without loss of generality, we only need to prove Theorem \ref{mainthm1} for $\|\varphi\|_{C^{2}(\partial\Omega)}=1$, and for the general case by considering $u/\|\varphi\|_{C^{2}(\partial\Omega)}$ if $\|\varphi\|_{C^{2}(\partial\Omega)}>0$. If $\varphi\big|_{\partial\Omega}=0$, then $u=0$.

\begin{prop}\label{prop_gradient}
Under the hypotheses of Theorem \ref{mainthm1}, and the normalization $\|\varphi\|_{C^{2}(\partial\Omega)}=1$, let $v_{i}^{\alpha}$ and $v_{0}$ be the solution to \eqref{v1alpha} and \eqref{v3}, respectively. Then for $0<\epsilon<1/2$, we have
\begin{align}
&\big\|\nabla{v}_{0}\big\|_{L^{\infty}(\widetilde{\Omega})}\leq\,C;\label{mainev0}\\
&\big\|\nabla(v_{1}^{\alpha}+v_{2}^{\alpha})\big\|_{L^{\infty}(\widetilde{\Omega})}
\leq\,C,\quad\alpha=1,2,3;\label{mainev1+23}\\
&\big|\nabla{v}_{i}^{\alpha}(x)\big|
\leq\frac{C}{\epsilon+\mathrm{dist}^{2}(x,\overline{P_{1}P_{2}})},\quad\,i=1,2,~~
\alpha=1,2,3,~~x\in\widetilde{\Omega};\label{mainev1}\\
&\big|\nabla{v}_{i}^{\alpha}(x)\big|
\leq\,\frac{C\mathrm{dist}(x,\overline{P_{1}P_{2}})}{\epsilon+\mathrm{dist}^{2}(x,\overline{P_{1}P_{2}})}+C,\quad\,i=1,2,
~~\alpha=4,5,6,~~x\in\widetilde{\Omega};\label{mainevi3}
\end{align}
and
\begin{equation}\label{maineC}
\big|C_{i}^{\alpha}\big|\leq\,C,\quad
i=1,2,~\alpha=1,2,\cdots,6;
\end{equation}
\begin{equation}\label{maineC1-C2}
\big|C_{1}^{\alpha}-C_{2}^{\alpha}\big|\leq\,\frac{C}{|\ln\epsilon|},\quad
\alpha=1,2,3.
\end{equation}
\end{prop}

\begin{proof}[Proof of Theorem \ref{mainthm1} by using Proposition \ref{prop_gradient}] Clearly, we only need to prove the theorem under the normalization $\|\varphi\|_{C^{2}(\partial\Omega)}=1$.

Since
$$\nabla{u}=
\begin{pmatrix}
0&C_{i}^{4}&C_{i}^{5}~\\\\
-C_{i}^{4}&0&C_{i}^{6}\\\\
-C_{i}^{5}&-C_{i}^{6}&0
\end{pmatrix},\quad\mbox{in}~D_{i},\quad\,i=1,2,$$
estimate \eqref{mainestimatesD1D2} follows from \eqref{maineC}.

By \eqref{nablau} and Proposition \ref{prop_gradient}, we have
\begin{align}\label{nablau_final}
\big|\nabla{u}(x)\big|&\leq\sum_{\alpha=1}^{3}\left|C_{1}^{\alpha}-C_{2}^{\alpha}\right|\left|\nabla{v}_{1}^{\alpha}(x)\right|
+\sum_{i=1}^{2}\sum_{\alpha=4}^{6}\left|C_{i}^{\alpha}\right|\left|\nabla{v}_{i}^{\alpha}\right|+C\nonumber\\
&\leq\frac{C}{|\ln\epsilon|\left(\epsilon+\mathrm{dist}^{2}(x,\overline{P_{1}P_{2}})\right)}
+\frac{C\mathrm{dist}(x,\overline{P_{1}P_{2}})}{\epsilon+\mathrm{dist}^{2}(x,\overline{P_{1}P_{2}})}.
\end{align}
Theorem \ref{mainthm1} follows immediately.
\end{proof}

To complete this section, we recall some properties of the tensor $\mathbb{C}$. For the isotropic elastic material, let
\begin{equation}\label{coeff_C}
\mathbb{C}:=(C_{ij\,kl})=\left(\lambda\delta_{ij}\delta_{kl}+\mu\left(\delta_{ik}\delta_{jl}+\delta_{il}\delta_{jk}\right)\right),
\quad\,\mu>0,\quad\,d\lambda+2\mu>0.
\end{equation}
The components $C_{ij\,kl}$ satisfy the following symmetric condition:
\begin{equation}\label{coeff1}
C_{ij\,kl}=C_{kl\,ij}=C_{klj\,i},\quad\,i,j,k,l=1,2,\cdots,d.
\end{equation}
We will use the following notations:
$$(\mathbb{C}A)_{ij}=\sum_{k,l=1}^{d}C_{ij\,kl}A_{kl},\quad\mbox{and}\quad
(A,B):=A:B=\sum_{i,j=1}^{d}A_{ij}B_{ij},$$
for every pair of $d\times{d}$ matrices $A=(A_{ij}),B=(B_{ij})$.
By the symmetric condition \eqref{coeff1}, we have
\begin{equation}\label{coeff3}
(\mathbb{C}A,B)=(A,\mathbb{C}B),
\end{equation}
$$(\mathbb{C}A,B)=(\mathbb{C}A^{T},B)=(\mathbb{C}A,C)=(\mathbb{C}A^{T},C).$$
For an arbitrary $d\times{d}$ real symmetric matrix $\eta=(\eta_{ij})$, we have
$$C_{ij\,kl}\,\eta_{kl}\eta_{ij}=\lambda\,\eta_{ii}\eta_{kk}+2\mu\,\eta_{kj}\eta_{kj}.$$
It follows from \eqref{coeff_C} that $\mathbb{C}$ satisfies the ellipticity condition
\begin{equation}\label{coeff2}
\min\Big\{2\mu,d\lambda+2\mu\Big\}|\eta|^{2}\leq\,C_{ij\,kl}\,\eta_{kl}\eta_{ij}\leq\,\max\Big\{2\mu,d\lambda+2\mu\Big\}|\eta|^{2},
\end{equation}
where $|\eta|^{2}=\sum\limits_{i,j=1}^{d}\eta_{ij}^{2}$. In particular,
\begin{equation}\label{coeff5}
\min\Big\{2\mu,d\lambda+2\mu\Big\}\left|A+A^{T}\right|^{2}\leq\left(\mathbb{C}\left(A+A^{T}\right),\left(A+A^{T}\right)\right).
\end{equation}

\section{Estimates of $|\nabla{v}_{0}|$, $|\nabla{v}_{i}^{\alpha}|$, and $|\nabla({v}_{1}^{\alpha}+{v}_{2}^{\alpha})|$ }\label{sec_gradient123}

We first fix notations. Use $(x_{1},x_{2},x_{3})$ to denote a point in $\mathbb{R}^{3}$ and $x'=(x_{1},x_{2})$. By a translation and rotation if necessary, we may assume without loss of generality that the points $P_{1}$ and $P_{2}$ in \eqref{P1P2} satisfy
$$P_{1}=\left(0',\frac{\epsilon}{2}\right)\in\partial{D}_{1},\quad\mbox{and}\quad\,P_{2}=\left(0',-\frac{\epsilon}{2}\right)\in\partial{D}_{2}.$$
Fix a small universal constant $R$, such that the portion of $\partial{D}_{1}$ and  $\partial{D}_{2}$ near $P_{1}$ and $P_{2}$, respectively, can be represented by
\begin{equation}\label{x3}
x_{3}=\frac{\epsilon}{2}+h_{1}(x'),\quad\mbox{and}\quad\,x_{3}=-\frac{\epsilon}{2}+h_{2}(x'), \quad\mbox{for}~~|x'|<2R.
\end{equation}
Then by the smoothness assumptions on $\partial{D}_{1}$ and $\partial{D}_{2}$, the functions $h_{1}(x')$ and $h_{2}(x')$ are of class $C^{2,\gamma}(B_{R}(0'))$,
satisfying
$$\frac{\epsilon}{2}+h_{1}(x')>-\frac{\epsilon}{2}+h_{2}(x'),\quad\mbox{for}~~|x'|<2R,$$
\begin{equation}\label{h1h20}
h_{1}(0')=h_{2}(0')=0,\quad\nabla{h}_{1}(0')=\nabla{h}_{2}(0')=0,
\end{equation}
\begin{equation}\label{h1h22}
\nabla^{2}h_{1}(0')\geq\kappa_{0}I,
\quad\,\nabla^{2}h_{2}(0')\leq-\kappa_{0}I,
\end{equation}
and
\begin{equation}\label{h1h2}
\|h_{1}\|_{C^{2,\gamma}(B_{2R}')}+\|h_{2}\|_{C^{2,\gamma}(B_{2R}')}\leq{C}.
\end{equation}
In particular, we only use a weaker relative strict convexity assumption of $\partial{D}_{1}$ and $\partial{D}_{2}$, that is
\begin{equation}\label{h1h23}
h_{1}(x')-h_{2}(x')\geq\kappa_{0}|x'|^{2},\quad\mbox{if}~~|x'|<2R.
\end{equation}

For $0\leq\,r\leq2R$, denote
$$\Omega_{r}:=\left\{~(x',x_{3})\in\mathbb{R}^{3}~\big|~-\frac{\epsilon}{2}+h_{2}(x')<x_{3}<\frac{\epsilon}{2}+h_{1}(x'),
~|x'|<r~\right\}.
$$
For $0\leq|z'|<R$, let
\begin{equation}\label{hatomega}
\widehat{\Omega}_{s}(z'):=\left\{~(x',x_{3})\in\mathbb{R}^{3}~\big|~-\frac{\epsilon}{2}+h_{2}(x')<x_{3}<\frac{\epsilon}{2}+h_{1}(x'),
~|x'-z'|<s~\right\}.
\end{equation}

\subsection{Estimates of $|\nabla{v}_{0}|$, $|\nabla{v}_{i}^{\alpha}|$ for $\alpha=1,2,3$, and $|\nabla(v_{1}^{\alpha}+v_{2}^{\alpha})|$}

\begin{lemma}\label{lem_v3_v1+v2}
\begin{equation}\label{estimate_v3}
\|{v}_{0}\|_{L^{\infty}(\widetilde{\Omega})}+\|\nabla{v}_{0}\|_{L^{\infty}(\widetilde{\Omega})}\leq\,C.
\end{equation}
\begin{equation}\label{estimate_v1+v2}
\|v_{1}^{\alpha}+v_{2}^{\alpha}\|_{L^{\infty}(\widetilde{\Omega})}
+\|\nabla(v_{1}^{\alpha}+v_{2}^{\alpha})\|_{L^{\infty}(\widetilde{\Omega})}\leq\,C,\quad\alpha=1,2,\cdots,6.
\end{equation}
\end{lemma}

The proof of Lemma \ref{lem_v3_v1+v2} is essentially the same as in \cite{bll} for dimension two. We omit it here. By Lemma \ref{lem_v3_v1+v2}, \eqref{mainev0} and \eqref{mainev1+23} is proved. 

To estimate \eqref{mainev1}, we introduce a scalar function $\bar{u}\in{C}^{2}(\mathbb{R}^{3})$, such that $\bar{u}=1$ on
$\partial{D}_{1}$, $\bar{u}=0$ on
$\partial{D}_{2}\cup\partial\Omega$,
\begin{align}\label{ubar}
\bar{u}(x)
=\frac{x_{3}-h_{2}(x')+\frac{\epsilon}{2}}{\epsilon+h_{1}(x')-h_{2}(x')},\quad\mbox{in}~~\Omega_{2R},
\end{align}
and
\begin{equation}\label{nablau_bar_outside}
\|\bar{u}\|_{C^{2}(\mathbb{R}^{3}\setminus\Omega_{R})}\leq\,C.
\end{equation}
Define
\begin{equation}\label{def:ubar1112}
\bar{u}_{1}^{\alpha}=\bar{u}\psi^{\alpha},\quad\alpha=1,2,3,
\quad\,\mbox{in}~~\widetilde{\Omega},
\end{equation}
then $\bar{u}_{1}^{\alpha}=v_{1}^{\alpha}$ on $\partial\widetilde{\Omega}$.

Similarly, we define
\begin{equation}\label{def:ubar2122}
\bar{u}_{2}^{\alpha}=\underline{u}\psi^{\alpha},\quad\alpha=1,2,3,
\quad\,\mbox{in}~~\widetilde{\Omega},
\end{equation}
such that $\bar{u}_{2}^{\alpha}=v_{2}^{\alpha}$ on $\partial\widetilde{\Omega}$, where $\underline{u}$ is a scalar function in ${C}^{2}(\mathbb{R}^{3})$ satisfying $\underline{u}=1$ on
$\partial{D}_{2}$, $\underline{u}=0$ on
$\partial{D}_{1}\cup\partial\Omega$,
\begin{align}\label{u_underline}
\underline{u}(x)
=\frac{-x_{3}+h_{1}(x')+\frac{\epsilon}{2}}{\epsilon+h_{1}(x')-h_{2}(x')},\quad\mbox{in}~~\Omega_{2R},
\end{align}
and
\begin{equation}\label{nablau_underline_outside}
\|\underline{u}\|_{C^{2}(\mathbb{R}^{3}\setminus\Omega_{R})}\leq\,C.
\end{equation}

In order to prove \eqref{mainev1}, it suffices to prove the following proposition.

\begin{prop}\label{prop1}
Assume the above, let $v_{i}^{\alpha}\in{H}^1(\widetilde{\Omega};\mathbb{R}^{3})$ be the
weak solution of \eqref{v1alpha} with $\alpha=1,2,3$. Then for $i=1,2,~\alpha=1,2,3$, 
\begin{equation}\label{energy_w}
\int_{\widetilde{\Omega}}\left|\nabla(v_{i}^{\alpha}-\bar{u}_{i}^{\alpha})\right|^{2}dx\leq\,C;
\end{equation}
and 
\begin{equation}\label{nablavialpha_outomega1/2}
\left\|\nabla{v}_{i}^{\alpha}\right\|_{L^{\infty}(\widetilde{\Omega}\setminus\Omega_{R})}\leq\,C,
\end{equation}
\begin{equation}\label{nabla_w_ialpha}
\left|\nabla(v_{i}^{\alpha}-\bar{u}_{i}^{\alpha})(x)\right|\leq
\begin{cases}
\displaystyle \frac{C}{\sqrt{\epsilon}},&|x'|\leq\sqrt{\epsilon},\\
\displaystyle \frac{C}{|x'|},&\sqrt{\epsilon}<|x'|\leq\,R,
\end{cases}\quad\forall~~x\in\Omega_{R}.
\end{equation}
Consequently,
\begin{equation}\label{nablav1}
\left|\nabla{v}_{i}^{\alpha}(x)\right|\leq\frac{C}{\epsilon+|x'|^{2}},\qquad\quad\forall~~x\in\Omega_{R},
\end{equation}
and
\begin{equation}\label{nablax'v1}
\left|\nabla_{x'}{v}_{i}^{\alpha}(x)\right|\leq\begin{cases}
\displaystyle \frac{C}{\sqrt{\epsilon}},&|x'|\leq\sqrt{\epsilon},\\
\displaystyle \frac{C}{|x'|},&\sqrt{\epsilon}<|x'|\leq\,R.
\end{cases}
\end{equation}
\end{prop}

A direct calculation gives, in view of \eqref{h1h20}-\eqref{h1h23}, that
\begin{equation}\label{nablau_bar}
|\partial_{x_{k}}\bar{u}(x)|\leq\frac{C|x_{k}|}{\epsilon+|x'|^{2}},~~k=1,2,\qquad
|\partial_{x_{3}}\bar{u}(x)|\leq\frac{C}{\epsilon+|x'|^{2}},\quad~x\in\Omega_{R}.
\end{equation}
Thus
\begin{equation}\label{nablau_123}
|\nabla\bar{u}_{i}^{\alpha}(x)|\leq\frac{C}{\epsilon+|x'|^{2}},\quad\,i=1,2,~\alpha=1,2,3,\quad~x\in\Omega_{R}.
\end{equation}
For $k,l=1,2$,
\begin{equation}\label{nabla2u_bar}
|\partial_{x_{k}x_{l}}\bar{u}(x)|\leq\frac{C}{\epsilon+|x'|^{2}},\quad
|\partial_{x_{k}x_{3}}\bar{u}(x)|\leq\frac{C|x'|}{(\epsilon+|x'|^{2})^{2}},
\quad\partial_{x_{3}x_{3}}\bar{u}(x)=0,\quad~x\in\Omega_{R}.
\end{equation}
For $\bar{u}_{i}^{\alpha}$, defined by \eqref{def:ubar1112} and \eqref{def:ubar2122}, making use of \eqref{L_u}
 and \eqref{nabla2u_bar}, we have, for $i=1,2,~\alpha=1,2,3$,
\begin{equation}\label{L_ubar_ialpha}
\left|\mathcal{L}_{\lambda,\mu}\bar{u}_{i}^{\alpha}(x)\right|\leq\,C\sum_{k+l<6}\left|\partial_{x_{k}x_{l}}\bar{u}(x)\right|\leq\frac{C}{\epsilon+|x'|^{2}}
+\frac{C|x'|}{(\epsilon+|x'|^{2})^{2}},\quad~x\in\Omega_{R}.
\end{equation}

For $|z'|\leq\,2R$, we always use $\delta$ to denote
\begin{equation}\label{delta}
\delta:=\delta(z')=\frac{\epsilon+h_{1}(z')-h_{2}(z')}{2}.
\end{equation}
By \eqref{h1h20}-\eqref{h1h23},
\begin{equation}
\frac{1}{C}\left(\epsilon+|z'|^{2}\right)\leq\delta(z')\leq\,C\left(\epsilon+|z'|^{2}\right).
\end{equation}

\begin{proof}[Proof of Proposition \ref{prop1}]
Let
\begin{equation}\label{def_w}
w_{i}^{\alpha}:=v_{i}^{\alpha}-\bar{u}_{i}^{\alpha},\qquad\,i=1,2,~\alpha=1,2,3.
\end{equation}
For simplicity, denote
$$w:=w_{i}^{\alpha},\qquad\mbox{and}\quad\tilde{u}=\bar{u}_{i}^{\alpha},\qquad\,i=1,2,~\alpha=1,2,3.$$
The proof is divided into four steps.

\noindent{\bf STEP 1.} Proof of \eqref{energy_w} and \eqref{nablavialpha_outomega1/2}.

By \eqref{def_w} and \eqref{v1alpha},
\begin{equation}\label{w20}
\begin{cases}
\mathcal{L}_{\lambda,\mu}w=-\mathcal{L}_{\lambda,\mu}\tilde{u},&\mbox{in}~\widetilde{\Omega},\\
w=0,&\mbox{on}~\partial\widetilde{\Omega}.
\end{cases}
\end{equation}
Multiplying the equation in \eqref{w20} by $w$ and integrating by parts, we have
\begin{align}\label{integratingbyparts}
\int_{\widetilde{\Omega}}\left(\mathbb{C}^{0}e(w),e(w)\right)dx
=\int_{\widetilde{\Omega}}w\left(\mathcal{L}_{\lambda,\mu}\tilde{u}\right)dx.
\end{align}
By the Poincar\'{e} inequality,
\begin{equation}
\|w\|_{L^{2}(\widetilde{\Omega}\setminus\Omega_{R})}\leq\,C\|\nabla{w}\|_{L^{2}(\widetilde{\Omega}\setminus\Omega_{R})}.
\end{equation}
Note that the above constant $C$ is independent of $\epsilon$. By the Sobolev trace embedding theorem,
\begin{align}\label{traceembedding}
\int\limits_{\scriptstyle |x'|=R,\atop\scriptstyle
-\epsilon/2+h_{2}(x')<x_{3}<\epsilon/2+h_{1}(x')\hfill}|w|&\leq\,C\left(\int_{\widetilde{\Omega}\setminus\Omega_{R}}\left|\nabla{w}\right|^{2}dx\right)^{1/2}.
\end{align}
It follows from the first Korn's inequality, \eqref{coeff5}, \eqref{integratingbyparts} and the definition of $\tilde{u}$ that
\begin{align}\label{energy_w_1}
\int_{\widetilde{\Omega}}\left|\nabla{w}\right|^{2}dx
&\leq\,2\int_{\widetilde{\Omega}}|e(w)|^{2}dx\nonumber\\
&\leq\,C\bigg|\int_{\Omega_{R}}w\left(\mathcal{L}_{\lambda,\mu}\tilde{u}\right)dx\bigg|
+C\bigg|\int_{\widetilde{\Omega}\setminus\Omega_{R}}w\left(\mathcal{L}_{\lambda,\mu}\tilde{u}\right)dx\bigg|\nonumber\\
&\leq\,C\bigg|\int_{\Omega_{R}}w\left(\mathcal{L}_{\lambda,\mu}\tilde{u}\right)dx\bigg|
+C\int_{\widetilde{\Omega}\setminus\Omega_{R}}|w|dx\nonumber\\
&\leq\,C\bigg|\int_{\Omega_{R}}w\left(\mathcal{L}_{\lambda,\mu}\tilde{u}\right)dx\bigg|
+C\left(\int_{\widetilde{\Omega}\setminus\Omega_{R}}|\nabla{w}|^{2}dx\right)^{1/2},
\end{align}
while, using \eqref{L_u} and \eqref{traceembedding},
\begin{align}\label{energy_w_2}
\left|\int_{\Omega_{R}}w\left(\mathcal{L}_{\lambda,\mu}\tilde{u}\right)dx\right|
&\leq\,C\sum_{k+l<6}\left|\int_{\Omega_{R}}w\partial_{x_{k}x_{l}}\tilde{u}\,dx\right|\nonumber\\
&\leq\,C\int_{\Omega_{R}}|\nabla{w}|\big|\nabla_{x'}\tilde{u}\big|dx
+\int\limits_{\scriptstyle |x'|=R,\atop\scriptstyle
-\epsilon/2+h_{2}(x')<x_{3}<\epsilon/2+h_{1}(x')\hfill}C~|w|\nonumber\\
&\leq\,C\left(\int_{\Omega_{R}}|\nabla{w}|^{2}dx\right)^{1/2}
\left(\int_{\Omega_{R}}\big|\nabla_{x'}\tilde{u}\big|^{2}dx\right)^{1/2}
+C\left(\int_{\widetilde{\Omega}\setminus\Omega_{R}}\left|\nabla{w}\right|^{2}dx\right)^{1/2}.
\end{align}
Using
$$\int_{\Omega_{R}}\big|\nabla_{x'}\tilde{u}\big|^{2}dx\leq\,C,$$
we have, from the above,
$$\int_{\widetilde{\Omega}}\left|\nabla{w}\right|^{2}dx
\leq\,C\left(\int_{\widetilde{\Omega}}|\nabla{w}|^{2}dx\right)^{1/2}.$$
This estimate yields \eqref{energy_w}.

A consequence of \eqref{energy_w} and \eqref{nablau_bar_outside} is
$$\int_{\widetilde{\Omega}\setminus\Omega_{R/2}}|\nabla{v}_{i}^{\alpha}|^{2}dx\leq2\int_{\widetilde{\Omega}\setminus\Omega_{R/2}}\Big(|\nabla\bar{u}_{i}^{\alpha}|^{2}+\left|\nabla(v_{i}^{\alpha}-\bar{u}_{i}^{\alpha})\right|^{2}\Big)dx\leq\,C.$$
Applying classical elliptic estimates, we obtain \eqref{nablavialpha_outomega1/2}.

\noindent{\bf STEP 2.} Proof of
\begin{align}\label{energy_w_inomega_z1}
\int_{\widehat{\Omega}_{\delta}(z')}\left|\nabla{w}\right|^{2}dx&\leq
\begin{cases}
\displaystyle C\epsilon^{2},&0\leq|z'|\leq\sqrt{\epsilon},\\
\displaystyle C|z'|^{4},&\sqrt{\epsilon}<|z'|\leq\,R.
\end{cases}
\end{align}

For $0<t<s<R$, let $\eta$ be a smooth function satisfying $0\leq\eta(x')\leq1$, $\eta(x')=1$ if $|x'-z'|<t$, $\eta(x')=0$ if $|x'-z'|>s$, and $|\nabla\eta(x')|\leq\frac{2}{s-t}$. Multiplying the equation in \eqref{w20} by $w\eta^{2}$ and integrating by parts leads to
\begin{equation}\label{I3-1}
\int_{\widehat{\Omega}_{s}(z')}\left(\mathbb{C}^{0}e(w),e(w\eta^{2})\right)dx
=\int_{\widehat{\Omega}_{s}(z')}(w\eta^{2})\mathcal{L}_{\lambda,\mu}\tilde{u}\,dx.
\end{equation}
For the left hand side of \eqref{I3-1}, using the first Korn's inequality and some standard arguments, we have
\begin{align*}
\int_{\widehat{\Omega}_{s}(z')}\left(\mathbb{C}^{0}e(w),e(w\eta^{2})\right)dx
\geq\frac{1}{C}\int_{\widehat{\Omega}_{s}(z')}|\nabla(w\eta)|^{2}dx
-C\int_{\widehat{\Omega}_{s}(z')}|w|^{2}|\nabla\eta|^{2}dx,
\end{align*}
and for the right hand side of \eqref{I3-1},
\begin{align*}
\bigg|\int_{\widehat{\Omega}_{s}(z')}(w\eta^{2})\mathcal{L}_{\lambda,\mu}\tilde{u}\,dx~\bigg|
&\leq\,\left(\int_{\widehat{\Omega}_{s}(z')}|w|^{2}dx\right)^{1/2}\left(\int_{\widehat{\Omega}_{s}(z')}\left|\mathcal{L}_{\lambda,\mu}\tilde{u}\right|^{2}dx\right)^{1/2}\\
&\leq\,\frac{1}{(s-t)^{2}}\int_{\widehat{\Omega}_{s}(z')}|w|^{2}dx
+(s-t)^{2}\int_{\widehat{\Omega}_{s}(z')}\left|\mathcal{L}_{\lambda,\mu}\tilde{u}\right|^{2}dx.
\end{align*}
It follows that
\begin{align}\label{FsFt11}
\int_{\widehat{\Omega}_{t}(z')}|\nabla{w}|^{2}dx\leq\,\frac{C}{(s-t)^{2}}\int_{\widehat{\Omega}_{s}(z')}|w|^{2}dx
+C(s-t)^{2}\int_{\widehat{\Omega}_{s}(z')}\left|\mathcal{L}_{\lambda,\mu}\tilde{u}\right|^{2}dx.
\end{align}

{\bf Case 1.} Estimate \eqref{energy_w_inomega_z1} for $\sqrt{\epsilon}\leq|z'|\leq\,R$.

Note that for $\sqrt{\epsilon}\leq|z'|\leq\,R$, $0<t<s<\frac{2|z'|}{3}$, we have
\begin{align}\label{energy_w_square}
\int_{\widehat{\Omega}_{s}(z')}|w|^{2}dx
&=\int_{|x'-z'|\leq\,s}\int_{-\frac{\epsilon}{2}+h_{2}(x')}^{\frac{\epsilon}{2}+h_{1}(x')}
|w(x',x_{3})|^{2}dx_{3}dx'\nonumber\\
&\leq\int_{|x'-z'|\leq\,s}(\epsilon+h_{1}(x')-h_{2}(x'))^{2}\int_{-\frac{\epsilon}{2}+h_{2}(x')}^{\frac{\epsilon}{2}+h_{1}(x')}
|\partial_{x_{3}}w(x',x_{3})|^{2}dx_{3}dx'\nonumber\\
&\leq\,C|z'|^{4}\int_{\widehat{\Omega}_{s}(z')}|\nabla{w}|^{2}dx.
\end{align}
By \eqref{L_ubar_ialpha}, we have
\begin{align}\label{integal_Lubar11}
\int_{\widehat{\Omega}_{s}(z')}\left|\mathcal{L}_{\lambda,\mu}\tilde{u}\right|^{2}dx
&\leq\int_{\widehat{\Omega}_{s}(z')}\left(\frac{C}{\epsilon+|x'|^{2}}+\frac{C|x'|}{(\epsilon+|x'|^{2})^{2}}\right)^{2}dx\nonumber\\
&\leq\,C\int_{|x'-z'|<s}\left(\frac{1}{\epsilon+|x'|^{2}}+\frac{|x'|^{2}}{(\epsilon+|x'|^{2})^{3}}\right)dx'\nonumber\\
&\leq\frac{Cs^{2}}{|z'|^{4}},\qquad\,0<s<\frac{2|z'|}{3}.
\end{align}
Denote
$$\widehat{F}(t):=\int_{\widehat{\Omega}_{t}(z')}|\nabla{w}|^{2}dx.$$
It follows from \eqref{FsFt11}, \eqref{energy_w_square} and \eqref{integal_Lubar11} that
\begin{equation}\label{tildeF111}
\widehat{F}(t)\leq\,\left(\frac{C_{0}|z'|^{2}}{s-t}\right)^{2}\widehat{F}(s)+C(s-t)^{2}\frac{s^{2}}{|z'|^{4}},
\qquad\forall~0<t<s<\frac{2|z'|}{3}.
\end{equation}
Set $t_{i}=\delta+2C_{0}i\,|z'|^{2}$, $i=0,1,2,\cdots$. Then
$$\frac{C_{0}|z'|^{2}}{t_{i+1}-t_{i}}=\frac{1}{2}.$$
Let $k=\left[\frac{1}{4C_{0}|z'|}\right]$. Using \eqref{tildeF111} with $s=t_{i+1}$ and $t=t_{i}$, we have
$$\widehat{F}(t_{i})\leq\,\frac{1}{4}\widehat{F}(t_{i+1})+\frac{C(t_{i+1}-t_{i})^{2}t_{i+1}^{2}}{|z'|^{4}}
\leq\,\frac{1}{4}\widehat{F}(t_{i+1})+C(i+1)^{2}|z'|^{4},\quad\,i=0,1,2,\cdots,k.
$$
After $k$ iterations, using \eqref{energy_w}, we obtain
\begin{align*}
\widehat{F}(t_{0})&\leq\,\Big(\frac{1}{4}\Big)^{k}\widehat{F}(t_{k})+C\sum_{l=1}^{k}\Big(\frac{1}{4}\Big)^{l-1}l^{2}|z'|^{4}\\
&\leq\,\Big(\frac{1}{4}\Big)^{k}\widehat{F}(\frac{2|z'|}{3})+C|z'|^{4}\sum_{l=1}^{k}\Big(\frac{1}{4}\Big)^{l-1}l^{2}
\leq\,C|z'|^{4}.
\end{align*}
This implies that
$$\int_{\widehat{\Omega}_{\delta}(z')}|\nabla{w}|^{2}dx\leq\,C|z'|^{4}.$$

{\bf Case 2.} Estimate \eqref{energy_w_inomega_z1} for $0\leq|z'|\leq\sqrt{\epsilon}$.

For $0\leq|z'|\leq\sqrt{\epsilon}$, $0<t<s<\sqrt{\epsilon}$, estimate \eqref{energy_w_square} becomes
 \begin{align}\label{energy_w_square_in}
\int_{\widehat{\Omega}_{s}(z')}|w|^{2}dx
\leq\,C\epsilon^{2}\int_{\widehat{\Omega}_{s}(z')}|\nabla{w}|^{2}dx,\quad\,0<s<\sqrt{\epsilon},
\end{align}
while estimate \eqref{integal_Lubar11} becomes
\begin{align}\label{integal_Lubar11_in}
\int_{\widehat{\Omega}_{s}(z')}\left|\mathcal{L}_{\lambda,\mu}\tilde{u}\right|^{2}
&\leq\int_{|x'-z'|<s}\left(\frac{C}{\epsilon+|x'|^{2}}+\frac{C|x'|^{2}}{(\epsilon+|x'|^{2})^{3}}\right)dx'
\leq\,\frac{Cs^{2}}{\epsilon^{2}}.
\end{align}
Estimate \eqref{tildeF111} becomes, in view of \eqref{FsFt11}, \eqref{energy_w_square_in} and \eqref{integal_Lubar11_in},
\begin{equation}\label{tildeF111_in}
\widehat{F}(t)\leq\,\left(\frac{C_{0}\epsilon}{s-t}\right)^{2}\widehat{F}(s)
+C(s-t)^{2}\frac{s^{2}}{\epsilon^{2}},
\quad\forall~0<t<s<\sqrt{\epsilon}.
\end{equation}
For $0\leq|z'|\leq\sqrt{\epsilon}$, let $t_{i}=\delta+2C_{0}i\epsilon$, $i=0,1,2,\cdots$. Thus
$$\frac{C_{0}\epsilon}{t_{i+1}-t_{i}}=\frac{1}{2}.$$
Let $k=\left[\frac{1}{4C_{0}\sqrt{\epsilon}}\right]$. By \eqref{tildeF111_in} with $s=t_{i+1}$ and $t=t_{i}$, we have
\begin{align*}
\widehat{F}(t_{i})&\leq\,\frac{1}{4}\widehat{F}(t_{i+1})+\frac{C\epsilon^{2}t_{i+1}^{2}}{\epsilon^{2}}
\leq\,\frac{1}{4}\widehat{F}(t_{i+1})+C(i+1)^{2}\epsilon^{2},\quad\,i=0,1,2,\cdots,k.
\end{align*}
After $k$ iterations, we obtain
\begin{align*}
\widehat{F}(t_{0})&\leq\,\Big(\frac{1}{4}\Big)^{k}\widehat{F}(t_{k})
+C\sum_{l=1}^{k}\Big(\frac{1}{4}\Big)^{l-1}l^{2}\epsilon^{2}\\
&\leq\,\Big(\frac{1}{4}\Big)^{k}\widehat{F}(\sqrt{\epsilon})
+C\epsilon^{2}
\leq\,C\epsilon^{2}.
\end{align*}
This implies
$$\int_{\widehat{\Omega}_{\delta}(z')}|\nabla{w}|^{2}dx\leq\,C\epsilon^{2}.$$

\noindent{\bf STEP 3.} Proof of \eqref{nabla_w_ialpha}.

Making a change of variables, for $0\leq|z'|\leq\,R$,
\begin{equation}\label{changeofvariant}
 \left\{
  \begin{aligned}
  &x'-z'=\delta y',\\
  &x_{3}=\delta y_{3},
  \end{aligned}
 \right.
\end{equation}
the region $\widehat{\Omega}_{\delta}(z')$, becomes $Q_{1}$, where
$$Q_{r}=\left\{y\in\mathbb{R}^{3}~\big|~-\frac{\epsilon}{2\delta}+\frac{1}{\delta}h_{2}(\delta{y}'+z')<y_{3}
<\frac{\epsilon}{2\delta}+\frac{1}{\delta}h_{1}(\delta{y}'+z'),~|y'|<r\right\},\quad\mbox{for}~~r\leq1,$$ and the top and bottom boundaries of $Q_{r}$ become
$$y_{3}=\hat{h}_{1}(y')=:\frac{1}{\delta}\left(\frac{\epsilon}{2}+h_{1}(\delta\,y'+z')\right),~~\mbox{and}
~~\,y_{3}=\hat{h}_{2}(y'):=\frac{1}{\delta}\left(-\frac{\epsilon}{2}+h_{2}(\delta\,y'+z')\right),\quad|y'|<1,$$
respectively. Thus
$$\hat{h}_{1}(0')-\hat{h}_{2}(0')=\frac{1}{\delta}\left(\epsilon+h_{1}(z')-h_{2}(z')\right)=2,$$
and, by \eqref{h1h20} and \eqref{h1h22}, for $|y'|\leq1$,
$$\left|\nabla\hat{h}_{1}(y')\right|+\left|\nabla\hat{h}_{2}(y')\right|\leq\,C(\delta+|z'|),
\qquad\left|\nabla^{2}\hat{h}_{1}(y')\right|+\left|\nabla^{2}\hat{h}_{2}(y')\right|\leq\,C\delta.$$
Since $R$ is small, $Q_{1}$ is essentially $B_{1}(0')\times(-1,1)$ as far as applications of the Sobolev embedding theory and classical $L^{p}$ estimates for elliptic systems are concerned.
Let
\begin{equation}\label{def_U}
U(y',y_{3}):=\tilde{u}(x',x_{3}),
\quad\,W(y',y_{3}):=w(x',x_{3}),\qquad\,y\in\,Q_{1}.
\end{equation}
By \eqref{w20},
\begin{align}
\mathcal{L}_{\lambda,\mu}W
=\mathcal{L}_{\lambda,\mu}U,
\quad\quad\,y\in{Q_{1}}.
\end{align}
where
$$\mathcal{L}_{\lambda,\mu}U=\delta^{2}\mathcal{L}_{\lambda,\mu}\tilde{u}.$$
Since $W=0$ on the top and
bottom boundaries of $Q_{1}$, we
have, by the Poincar\'{e} inequality,
$$\left\|W\right\|_{H^{1}(Q_{1})}\leq\,C\left\|\nabla{W}\right\|_{L^{2}(Q_{1})}.$$
Using the interior and boundary $W^{2,p}$ estimates (see \cite{adn}, and Theorem 2.5 in \cite{gi})
and the Sobolev embedding theorem, we have, for some $p>3$,
\begin{align*}
\left\|\nabla{W}\right\|_{L^{\infty}(Q_{1/2})}\leq\,C\left\|W\right\|_{W^{2,p}(Q_{1/2})}
\leq\,C\left(\left\|\nabla{W}\right\|_{L^{2}(Q_{1})}+\left\|\mathcal{L}_{\lambda,\mu}U\right\|_{L^{\infty}(Q_{1})}\right),
\end{align*}
where $C$ depends only on $p$ and $Q_{1}$, but not on $\epsilon$. Thus
\begin{equation}\label{nablaw11}
\left\|\nabla{w}\right\|_{L^{\infty}(\widehat{\Omega}_{\delta/2}(z'))}\leq\,
\frac{C}{\delta}\left(\delta^{-\frac{1}{2}}\left\|\nabla{w}\right\|_{L^{2}(\widehat{\Omega}_{\delta}(z'))}
+\delta^{2}\left\|\mathcal{L}_{\lambda,\mu}\tilde{u}\right\|_{L^{\infty}(\widehat{\Omega}_{\delta}(z'))}\right).
\end{equation}

{\bf Case 1.} \eqref{nabla_w_ialpha} for $\sqrt{\epsilon}\leq|z'|\leq\,R$.

By \eqref{energy_w_inomega_z1},
$$\left\|\nabla{w}\right\|_{L^{2}(\widehat{\Omega}_{\delta}(z'))}^{2}=\int_{\widehat{\Omega}_{\delta}(z')}\left|\nabla{w}\right|^{2}dx\leq\,C|z'|^{4}.$$
By \eqref{L_ubar_ialpha},
$$\delta\left\|\mathcal{L}_{\lambda,\mu}\tilde{u}\right\|_{L^{\infty}(\widehat{\Omega}_{\delta}(z'))}
\leq\delta\left(\frac{C}{\epsilon+|z'|^{2}}+\frac{C|z'|}{(\epsilon+|z'|^{2})^{2}}\right)\leq\frac{C}{|z'|}.$$
It follows from \eqref{nablaw11} that
\begin{align*}\left|\nabla{w}(z',z_{3})\right|&\leq\frac{C|z'|^{2}}{\delta^{3/2}}
+\frac{C}{|z'|}\leq\frac{C}{|z'|},\qquad\forall~\sqrt{\epsilon}\leq|z'|\leq\,R.
\end{align*}

{\bf Case 2.} \eqref{nabla_w_ialpha} for $0\leq|z'|\leq\sqrt{\epsilon}$.

Using \eqref{energy_w_inomega_z1}, and \eqref{L_ubar_ialpha}, we have
$$\left\|\nabla{w}\right\|_{L^{2}(\widehat{\Omega}_{\delta}(z'))}
\leq\,C\epsilon,
\qquad\delta\left\|\mathcal{L}_{\lambda,\mu}\tilde{u}\right\|_{L^{\infty}(\widehat{\Omega}_{\delta}(z'))}\leq\frac{C|z'|}{\epsilon}+C,$$
and, using \eqref{nablaw11},
$$\left|\nabla{w}(z',z_{3})\right|
\leq\frac{C\epsilon}{\delta^{3/2}}+\frac{C|z'|}{\epsilon}+C
\leq\frac{C}{\sqrt{\epsilon}},
\qquad~\forall~~0\leq|z'|\leq\sqrt{\epsilon}.$$

\noindent{\bf STEP 4.} Proof of \eqref{nablav1} and \eqref{nablax'v1}.

Estimate \eqref{nablav1} and \eqref{nablax'v1} in $\Omega_{R}$ follows from \eqref{nabla_w_ialpha} and \eqref{nablau_bar}. 

Proposition \ref{prop1} is established.
\end{proof}

\subsection{Estimates of $|\nabla{v}_{i}^{\alpha}|$, $\alpha=4,5,6$}

Define
\begin{equation}\label{def:ubar1323}
\begin{aligned}
\bar{u}_{1}^{\alpha}=\bar{u}\psi^{\alpha},\quad\mbox{and}\quad\bar{u}_{2}^{\alpha}=\underline{u}\psi^{\alpha},
\qquad\alpha=4,5,6,
\quad\,\mbox{in}~~\widetilde{\Omega}.
\end{aligned}
\end{equation}
Clearly, $v_{i}^{\alpha}=\bar{u}_{i}^{\alpha}$ on $\partial\widetilde{\Omega}$, $i=1,2$, $\alpha=4,5,6$.

\begin{prop}\label{prop3}
Assume the above, let $v_{i}^{\alpha}\in{H}^1(\widetilde{\Omega};\mathbb{R}^{3})$ be the
weak solution of \eqref{v1alpha} with $\alpha=4,5,6$. Then for $i=1,2,~\alpha=4,5,6$,
\begin{equation}\label{lem31.02}
\int_{\widetilde{\Omega}}\left|{v}_{i}^{\alpha}\right|^{2}dx+
\int_{\widetilde{\Omega}}\left|\nabla{v}_{i}^{\alpha}\right|^{2}dx\leq\,C,
\end{equation}
and
\begin{equation}\label{nablavi3_outomega1/2}
\|\nabla{v}_{i}^{\alpha}\|_{L^{\infty}(\widetilde{\Omega}\setminus\Omega_{R})}\leq\,C,
\end{equation}
\begin{equation}\label{nabla_w13}
|\nabla(v_{i}^{\alpha}-\bar{u}_{i}^{\alpha})(x',x_{3})|\leq\,C,\qquad\,x\in\Omega_{R}.
\end{equation}
Consequently,
\begin{equation}\label{nabla_v13}
|\nabla{v}_{i}^{\alpha}(x',x_{3})|\leq\,\frac{C|x'|}{\epsilon+|x'|^{2}}+C,\qquad\,x\in\Omega_{R}.
\end{equation}
\end{prop}

Using \eqref{nablau_bar} and \eqref{nablau_bar_outside}, we have
\begin{equation}\label{nabla_baru14_in}
\left|\nabla\bar{u}_{i}^{\alpha}(x)\right|\leq\,\frac{C|x'|}{\epsilon+|x'|^{2}}+C,\quad\,x\in\Omega_{R},
\end{equation}
and
\begin{equation}\label{nabla_baru13_out}
\left|\nabla\bar{u}_{i}^{\alpha}(x)\right|\leq\,C,\quad\,x\in\widetilde{\Omega}\setminus\Omega_{R}.
\end{equation}
It follows from \eqref{def:ubar1323}, \eqref{L_u}, \eqref{nablau_bar} and \eqref{nabla2u_bar} that, for $i=1,2$, $\alpha=4,5,6$,
\begin{equation}\label{L_ubar_i3}
|\mathcal{L}_{\lambda,\mu}\bar{u}_{i}^{\alpha}|
\leq\,C\left(|\nabla\bar{u}|+(\epsilon+|x'|)\sum_{k+l<6}|\partial_{x_{k}x_{l}}\bar{u}|\right)
\leq\frac{C}{\epsilon+|x'|^{2}},\quad\,x\in\Omega_{R}.
\end{equation}

\begin{proof}[Proof of Proposition \ref{prop3}]
Denote
\begin{equation}\label{def_w456}
w_{i}^{\alpha}:=v_{i}^{\alpha}-\bar{u}_{i}^{\alpha},\qquad\,i=1,2,~\alpha=4,5,6.
\end{equation}
For simplicity, we also use the notation 
$$w:=w_{i}^{\alpha},\qquad \tilde{u}:=\bar{u}_{i}^{\alpha},\quad i=1,2,~~\alpha=4,5,6.$$ 
The proof is divided into three steps.

\noindent{\bf STEP 1.} Proof of \eqref{lem31.02} and \eqref{nablavi3_outomega1/2}. 

Similarly as Step 1 in the proof of Proposition \ref{prop1}, by \eqref{def_w456} and \eqref{v1alpha} with $\alpha=4,5,6$, Using \eqref{def:ubar1323}, and \eqref{L_u}, \eqref{traceembedding} again, \eqref{energy_w_2} is replaced by
\begin{align}\label{energy_w_24}
\int_{\Omega_{R}}w\left(\mathcal{L}_{\lambda,\mu}\tilde{u}\right)dx
&\leq\,C\int_{\Omega_{R}}|\nabla{w}||\nabla\tilde{u}|dx
+\int\limits_{\scriptstyle |x'|=R,\atop\scriptstyle
-\epsilon/2+h_{2}(x')<x_{3}<\epsilon/2+h_{1}(x')\hfill}C~|w|\nonumber\\
&\leq\,C\left(\int_{\Omega_{R}}|\nabla{w}|^{2}dx\right)^{1/2}
\left(\int_{\Omega_{R}}|\nabla\tilde{u}|^{2}dx\right)^{1/2}
+C\left(\int_{\widetilde{\Omega}\setminus\Omega_{R}}\left|\nabla{w}\right|^{2}dx\right)^{1/2}.
\end{align}
Using \eqref{nabla_baru14_in}, we have
\begin{equation}\label{energytildeu}
\int_{\Omega_{R}}|\nabla\tilde{u}|^{2}dx\leq\,C.
\end{equation}
It follows from \eqref{energy_w_1} for this situation that
$$\int_{\widetilde{\Omega}}\left|\nabla{w}\right|^{2}dx
\leq\,C\left(\int_{\widetilde{\Omega}}|\nabla{w}|^{2}dx\right)^{1/2}.$$
This implies
$$\int_{\widetilde{\Omega}}\left|\nabla{w}\right|^{2}dx
\leq\,C.$$
By the Poincar\'{e} inequality,
$$\int_{\widetilde{\Omega}}\left|w\right|^{2}dx+\int_{\widetilde{\Omega}}\left|\nabla{w}\right|^{2}dx
\leq\,C.$$
Combining with \eqref{energytildeu}, we obtain \eqref{lem31.02}. 

Using \eqref{lem31.02} and recalling the definition of $\tilde{u}$, we apply the standard elliptic estimates (see \cite{adn}) to obtain \eqref{nablavi3_outomega1/2}.

\noindent{\bf STEP 2.} Proof of 
\begin{equation}\label{energyw13_out}
\int_{\widehat{\Omega}_{\delta}(z')}|\nabla{w}|^{2}dx\leq
\begin{cases}
C|z'|^{6},&\sqrt{\epsilon}\leq|z'|<R,\\
C\epsilon^{3},&0\leq|z'|<\sqrt{\epsilon},
\end{cases}
\end{equation}
with $\delta=\delta(z')$ defined in \eqref{delta}.

The proof is similar to that of \eqref{energy_w_inomega_z1}.
We still have \eqref{FsFt11}.

{\bf Case 1.} Estimate \eqref{energyw13_out} for $\sqrt{\epsilon}\leq|z'|\leq\,R$.

For $0<t<s<\frac{2|z'|}{3}$, using \eqref{L_ubar_i3}, we have, instead of \eqref{integal_Lubar11},
\begin{align}\label{Lu13}
\int_{\widehat{\Omega}_{s}(z')}\left|\mathcal{L}_{\lambda,\mu}\tilde{u}\right|^{2}dx
\leq\int_{|x'-z'|<s}\frac{C}{\epsilon+|x'|^{2}}\,dx'\leq\frac{Cs^{2}}{|z'|^{2}}.
\end{align}
Using \eqref{energy_w_square}, instead of \eqref{tildeF111}, we have
\begin{equation}\label{tildeF13}
\widehat{F}(t)\leq\,\left(\frac{C_{0}|z'|^{2}}{s-t}\right)^{2}\widehat{F}(s)+C(s-t)^{2}\frac{s^{2}}{|z'|^{2}},
\quad\forall~0<t<s<\frac{2|z'|}{3}.
\end{equation}
We define $\{t_{i}\}$, $k$ and iterate as in the proof of \eqref{energy_w_inomega_z1}, right below formula \eqref{tildeF111}, to obtain
$$\widehat{F}(t_{0})
\leq\,\Big(\frac{1}{4}\Big)^{k}\widehat{F}(\frac{3|z'|}{2})+C|z'|^{6}\sum_{l=1}^{k}\Big(\frac{1}{4}\Big)^{l-1}l^{2}\leq\,C|z'|^{6}.$$
This implies that
$$\int_{\widehat{\Omega}_{\delta}(z')}|\nabla{w}|^{2}dx\leq\,C|z'|^{6}.$$

{\bf Case 2.} Estimate \eqref{energyw13_out} for $0\leq|z'|\leq\sqrt{\epsilon}$.

For $0<t<s<\sqrt{\epsilon}$, estimate \eqref{energy_w_square_in} remains the same. Estimate \eqref{integal_Lubar11_in} becomes
\begin{align}\label{integal_Lubar13_in}
\int_{\widehat{\Omega}_{s}(z')}\left|\mathcal{L}_{\lambda,\mu}\tilde{u}\right|^{2}dx
\leq\int_{|x'-z'|<s}\frac{C}{\epsilon+|x'|^{2}}\,dx'\leq\frac{Cs^{2}}{\epsilon},\qquad\,0<s<\sqrt{\epsilon}.
\end{align}
Estimate \eqref{tildeF111_in} becomes
\begin{equation}\label{tildeF2}
F(t)\leq\,\left(\frac{C_{0}\epsilon}{s-t}\right)^{2}F(s)+\frac{C(s-t)^{2}s^{2}}{\epsilon},
\quad\forall~0<t<s<\sqrt{\epsilon}.
\end{equation}
Define $\{t_{i}\}$, $k$ and iterate as in the proof of \eqref{energy_w_inomega_z1}, right below formula \eqref{tildeF111_in}, to obtain
$$F(t_{0})\leq\,\Big(\frac{1}{4}\Big)^{k}\widehat{F}(\sqrt{\epsilon})
+C\sum_{l=1}^{k}\Big(\frac{1}{4}\Big)^{l-1}l^{2}\epsilon^{3}
\leq\,C\epsilon^{3}.$$
This implies as before that
$$\int_{\widehat{\Omega}_{s}(z')}|\nabla{w}|^{2}dx
\leq\,C\epsilon^{3}.$$
\eqref{energyw13_out} is proved.

\noindent{\bf STEP 3.} Proof of \eqref{nabla_w13} and \eqref{nabla_v13}.

The proof  is similar to that of \eqref{nabla_w_ialpha}. In Case 1, for $\sqrt{\epsilon}\leq|z'|\leq\,R$, using estimates \eqref{energyw13_out} and \eqref{L_ubar_i3},
$$\int_{\widehat{\Omega}_{\delta}(z')}\left|\nabla{w}\right|^{2}dx\leq\,C|z'|^{6},
\quad\mbox{and}\quad
\delta\left\|\mathcal{L}_{\lambda,\mu}\tilde{u}\right\|_{L^{\infty}(\widehat{\Omega}_{\delta}(z'))}\leq\,C,$$
we obtain, using \eqref{nablaw11},
$$\left|\nabla{w}(z',z_{3})\right|\leq\frac{C|z'|^{3}}{\delta^{3/2}}+C
\leq\,C,\quad\mbox{for}~~\sqrt{\epsilon}\leq|z'|\leq\,R.$$

In Case 2, for $0\leq|z'|\leq\sqrt{\epsilon}$, using estimates \eqref{energyw13_out} and \eqref{L_ubar_i3},
$$\int_{\widehat{\Omega}_{\delta}(z')}\left|\nabla{w}\right|^{2}dx
\leq\,C\epsilon^{3},\quad\mbox{and}\quad
\delta\left\|\mathcal{L}_{\lambda,\mu}\tilde{u}\right\|_{L^{\infty}(\widehat{\Omega}_{\delta}(z'))}\leq\,C,$$
we have, using again \eqref{nablaw11},
$$\left|\nabla{w}(z',z_{3})\right|
\leq\frac{C\epsilon^{3/2}}{\delta^{3/2}}+C
\leq\,C,
\quad\mbox{for}~~0\leq|z'|\leq\sqrt{\epsilon}.$$
Estimate \eqref{nabla_w13} is established. 

Estimate \eqref{nabla_v13} follows from \eqref{nabla_w13} and \eqref{nabla_baru14_in}.
\end{proof}

\section{Estimates of $|C_{i}|$ and $|C_{1}^{\alpha}-C_{2}^{\alpha}|$, $\alpha=1,2,3$}\label{sec_C1C2}

In this section, we first prove that $C_{1}^{\alpha}$ and $C_{2}^{\alpha}$ are uniformly bounded with respect to $\epsilon$, and then estimate the difference  $C_{1}^{\alpha}-C_{2}^{\alpha}$.

\begin{prop}\label{prop_C1-C2}
Let $C_{i}^{\alpha}$ be defined in \eqref{decom_u}. Then
\begin{equation}\label{C1C2bound}
|C_{i}^{\alpha}|\leq\,C,\qquad\forall~i=1,2,~\alpha=1,2,\cdots,6;
\end{equation}
and
\begin{equation}\label{C1-C2}
\left|C_{1}^{\alpha}-C_{2}^{\alpha}\right|\leq\,\frac{C}{|\ln\epsilon|},\quad\alpha=1,2,3.
\end{equation}
\end{prop}

\subsection{Boundedness of $|C_{i}|$}

\begin{proof}[Proof of \eqref{C1C2bound}]
Let $u_{\epsilon}$ be the solution of \eqref{mainequation}. By theorem 4.6 in the appendix in \cite{bll}, $u_{\epsilon}$ is the minimizer of
$$I_{\infty}[u]:=\frac{1}{2}\int_{\widetilde{\Omega}}\Big(\mathbb{C}e(u),e(u)\Big)$$
on $\mathcal{A}$. It follows that
$$\|u_{\epsilon}\|_{H^{1}(\widetilde{\Omega})}^{2}\leq\,C\|e(u_{\epsilon})\|_{L^{2}(\widetilde{\Omega})}^{2}\leq\,CI_{\infty}[u_{\epsilon}]\leq\,C.$$
By the Sobolev trace embedding theorem,
$$\|u_{\epsilon}\|_{L^{2}(\partial{D}_{1}\cap\,B_{R})}\leq\,C.$$
Recalling that 
$$u_{\epsilon}=\sum_{\alpha=1}^{6}C_{1}^{\alpha}\psi^{\alpha},\qquad\mbox{on}~~\partial{D}_{1}.$$
If $C_{1}:=(C_{1}^{1},C_{1}^{2},\cdots,C_{1}^{6})^{T}=0$, there is nothing to prove. Otherwise
\begin{equation}\label{b3-1}
C\geq|C_{1}|\left\|\sum_{\alpha=1}^{6}\widehat{C}_{1}^{\alpha}\psi^{\alpha}\right\|_{L^{2}(\partial{D}_{1}\cap\,B_{R})},
\end{equation}
where $\widehat{C}_{1}^{\alpha}=\frac{C_{1}^{\alpha}}{|C_{1}|}$ and $|\widehat{C}_{1}|=1$. It is easy to see that
\begin{equation}\label{b3-2}
\left\|\sum_{\alpha=1}^{6}\widehat{C}_{1}^{\alpha}\psi^{\alpha}\right\|_{L^{2}(\partial{D}_{1}\cap\,B_{R})}\geq\frac{1}{C}.
\end{equation}
Indeed, if not, along a subsequence $\epsilon\rightarrow0$, $\widehat{C}_{1}^{\alpha}\rightarrow\overline{C}_{1}^{\alpha}$, and
$$\left\|\sum_{\alpha=1}^{6}\overline{C}_{1}^{\alpha}\psi^{\alpha}\right\|_{L^{2}(\partial{D}^{*}_{1}\setminus{B}_{R})}=0,$$
where $\partial{D}^{*}_{1}$ is the limit of $\partial{D}_{1}$ as $\epsilon\rightarrow0$ and $|\overline{C}_{1}|=1$. This implies $\sum_{\alpha=1}^{6}\overline{C}_{1}^{\alpha}\psi^{\alpha}=0$ on $\partial{D}^{*}_{1}\setminus{B}_{R}$. But $\left\{\psi^{\alpha}\big|_{\partial{D}^{*}_{1}\setminus{B}_{R}}\right\}$ is easily seen to be linear independent, using Lemma \ref{lem_xi}, we must have $\overline{C}_{1}=0$. This is a contradiction.  \eqref{C1C2bound} follows from \eqref{b3-1} and \eqref{b3-2}.
\end{proof}

\subsection{Estimates of $\left|C_{1}^{\alpha}-C_{2}^{\alpha}\right|$, $\alpha=1,2,3$}

In the rest of this section, we prove \eqref{C1-C2}. By the linearity of $e(u)$,
$$e(u)=\sum_{\alpha=1}^{6}C_{1}^{\alpha}e\left(v_{1}^{\alpha}\right)+\sum_{\alpha=1}^{6}C_{2}^{\alpha}e(v_{2}^{\alpha})+e({v}_{0}),
\quad\mbox{in}~~\widetilde{\Omega}.$$
It follows from the forth line of \eqref{mainequation} that
\begin{align}\label{C1C2_2}
\sum_{\alpha=1}^{6}C_{1}^{\alpha}\int_{\partial{D}_{j}}\frac{\partial{v}_{1}^{\alpha}}{\partial\nu_{0}}\bigg|_{+}\cdot\psi^{\beta}
&+\sum_{\alpha=1}^{6}C_{2}^{\alpha}\int_{\partial{D}_{j}}\frac{\partial{v}_{2}^{\alpha}}{\partial\nu_{0}}\bigg|_{+}\cdot\psi^{\beta}
+\int_{\partial{D}_{j}}\frac{\partial{v}_{0}}{\partial\nu_{0}}\bigg|_{+}\cdot\psi^{\beta}=0,\nonumber\\
&\quad\,j=1,2,~~\beta=1,2,\cdots,6.
\end{align}

Denote
\begin{align*}
a_{ij}^{\alpha\beta}=-\int_{\partial{D}_{j}}\frac{\partial{v}_{i}^{\alpha}}{\partial\nu_{0}}\bigg|_{+}\cdot\psi^{\beta},\qquad
b_{j}^{\beta}=\int_{\partial{D}_{j}}\frac{\partial{v}_{0}}{\partial\nu_{0}}\bigg|_{+}\cdot\psi^{\beta},\quad\,i,j=1,2,~~\alpha,\beta=1,2,\cdots,6.
\end{align*}
Multiplying the first line of \eqref{v1alpha} and \eqref{v3}, respectively, by $v_{j}^{\beta}$, and integrating by parts over $\widetilde{\Omega}$ leads to, in view of \eqref{equ4-1}, that
$$a_{ij}^{\alpha\beta}=\int_{\widetilde{\Omega}}\left(\mathbb{C}^{0}e(v_{i}^{\alpha}),e(v_{j}^{\beta})\right)dx,\quad
b_{j}^{\beta}=-\int_{\widetilde{\Omega}}\left(\mathbb{C}^{0}e({v}_{0}),e(v_{j}^{\beta})\right)dx.$$
Then \eqref{C1C2_2} can be written as
\begin{equation}\label{C1C2_3}
\left\{
\begin{aligned}
\sum_{\alpha=1}^{6}C_{1}^{\alpha}a_{11}^{\alpha\beta}+\sum_{\alpha=1}^{6}C_{2}^{\alpha}a_{21}^{\alpha\beta}-b_{1}^{\beta}&=0,\\
\sum_{\alpha=1}^{6}C_{1}^{\alpha}a_{12}^{\alpha\beta}+\sum_{\alpha=1}^{6}C_{2}^{\alpha}a_{22}^{\alpha\beta}-b_{2}^{\beta}&=0,
\end{aligned}
\right.\quad\quad~~\beta=1,2,\cdots,6.
\end{equation}
For simplicity, we use $a_{ij}$ to denote the $6\times6$ matrix $(a_{ij}^{\alpha\beta})$. To estimate $\left|C_{1}^{\alpha}-C_{2}^{\alpha}\right|$, $\alpha=1,2,3$, we only use the first six equations in \eqref{C1C2_3}:
\begin{equation}\label{equ_abc}
 a_{11}C_{1} +a_{21}C_{2} =b_{1},
\end{equation}
where
$$C_{1}=\Big(C_{1}^{1}, C_{1}^{2},\cdots, C_{1}^{6}\Big)^{T} ,
\quad\,C_{2}=\Big(C_{2}^{1}, C_{2}^{2},\cdots, C_{2}^{6}\Big)^{T},
\quad\,b_{1}=\Big(b_{1}^{1},b_{1}^{2},\cdots, b_{1}^{6}\Big)^{T}.$$
Set
\begin{equation*}
p:=b_{1}-\Big(a_{11}+a_{21}\Big)C_{2},
\end{equation*}
\eqref{equ_abc} can be rewritten as
\begin{equation}\label{C1C2_4}
a_{11}\Big(C_{1}-C_{2}\Big)=p.
\end{equation}

In order to prove \eqref{C1-C2}, we first estimate the right hand side of \eqref{C1C2_4}.
\begin{lemma}\label{lem_ap}
\begin{align*}
&\left|a_{11}^{\alpha\beta}+a_{21}^{\alpha\beta}\right|\leq\,C,\quad\alpha,\beta=1,2,\cdots,6;\\\\
&\left|b_{1}^{\beta}\right|\leq\,C,\quad\beta=1,2,\cdots,6.
\end{align*}
Consequently,
\begin{equation}\label{p_bound}
|p|\leq\,C.
\end{equation}
\end{lemma}

\begin{proof}
For $\beta=1,2,3$, using \eqref{nablavialpha_outomega1/2} and \eqref{nablav1}, we have
\begin{align}\label{nablav11123_L1}
\int_{\widetilde{\Omega}}\left|\nabla{v}_{1}^{\beta}\right|dx
\leq\int_{\Omega_{R}}\left|\nabla{v}_{1}^{\beta}\right|dx+\int_{\widetilde{\Omega}\setminus\Omega_{R}}\left|\nabla{v}_{1}^{\beta}\right|dx
\leq\int_{\Omega_{R}}\frac{Cdx}{\epsilon+|x'|^{2}}+C
\leq\,C.
\end{align}
For $\beta=4,5,6$, using \eqref{nablavi3_outomega1/2} and \eqref{nabla_v13}, we have
\begin{align}\label{nablav11456_L1}
\int_{\widetilde{\Omega}}\left|\nabla{v}_{1}^{\beta}\right|dx
\leq\int_{\Omega_{R}}\frac{C(\epsilon+|x'|)dx}{\epsilon+|x'|^{2}}+C
\leq\,C.
\end{align}
For $\alpha,\beta=1,2,\cdots,6$, using \eqref{mainev1+23}, \eqref{nablav11123_L1} and \eqref{nablav11456_L1}, we have
\begin{align*}
\left|a_{11}^{\alpha\beta}+a_{21}^{\alpha\beta}\right|
&=\left|\int_{\widetilde{\Omega}}\left(\mathbb{C}^{0}e(v_{1}^{\alpha}+v_{2}^{\alpha}),e(v_{1}^{\beta})\right)dx\right|
\leq\,C\left\|\nabla(v_{1}^{\alpha}+v_{2}^{\alpha})\right\|_{L^{\infty}(\widetilde{\Omega})}\int_{\widetilde{\Omega}}\left|\nabla{v}_{1}^{\beta}\right|dx
\leq\,C.
\end{align*}
Similarly, it follows from \eqref{mainev0}, \eqref{nablav11123_L1} and \eqref{nablav11456_L1} that
$$\left|b_{1}^{\beta}\right|=\left|\int_{\widetilde{\Omega}}\left(\mathbb{C}^{0}e(v_{1}^{\beta}),e({v}_{0})\right)dx\right|
\leq\,C\|\nabla{v}_{0}\|_{L^{\infty}(\widetilde{\Omega})}\int_{\widetilde{\Omega}}\left|\nabla{v}_{1}^{\beta}\right|dx
\leq\,C,\quad\beta=1,2,\cdots,6.$$
These estimates above, combining with \eqref{C1C2bound}, yield \eqref{p_bound}.
\end{proof}

It can be proved that $a_{11}$ is positive definite and therefore, recalling \eqref{C1C2_4},
$$C_{1}-C_{2}=(a_{11})^{-1}p.$$
Given \eqref{p_bound}, estimate \eqref{C1-C2} would follow from the above if $\|(a_{11})^{-1}\|\leq\frac{C}{|\ln\epsilon|}$. However $\|(a_{11})^{-1}\|\geq\frac{1}{C}>0$. We need to make more delicate estimate as below.

In view of the symmetry of $a_{11}^{\alpha\beta}$, we write it as a block matrix
$$
a_{11}=
\begin{pmatrix}
  A & B ~\\\\
  C & D ~\\
\end{pmatrix},
$$
where
$$A=\begin{pmatrix}
      a^{11}_{11} & a_{11}^{12} & a_{11}^{13} ~\\\\
      a_{11}^{21} & a_{11}^{22} & a_{11}^{23} ~\\\\
      a_{11}^{31} & a_{11}^{32} & a_{11}^{33} ~\\
    \end{pmatrix},
\quad
B=\begin{pmatrix}
      a_{11}^{14} & a_{11}^{15} & a_{11}^{16} ~\\\\
      a_{11}^{24} & a_{11}^{25} & a_{11}^{26} ~\\\\
      a_{11}^{34} & a_{11}^{35} & a_{11}^{36} ~\\
    \end{pmatrix},\quad\mbox{and}~
D=\begin{pmatrix}
      a_{11}^{44} & a_{11}^{45} & a_{11}^{46} ~\\\\
      a_{11}^{54} & a_{11}^{55} & a_{11}^{56} ~\\\\
      a_{11}^{64} & a_{11}^{65} & a_{11}^{66} ~\\
    \end{pmatrix}.
$$

\begin{lemma}\label{lem_a_11}
$a_{11}$ is positive definite, and
\begin{equation}\label{a11_1122}
\frac{|\ln\epsilon|}{C}\leq\,a_{11}^{\alpha\alpha}\leq\,C|\ln\epsilon|,\qquad\alpha=1,2,3;
\end{equation}
\begin{equation}\label{a11_456}
\frac{1}{C}\leq\,a_{11}^{\alpha\alpha}\leq\,C,\quad\alpha=4,5,6;
\end{equation}
and
\begin{equation}\label{a11_12}
\left|a_{11}^{\alpha\beta}\right|=\left|a_{11}^{\beta\alpha}\right|\leq\,C,\qquad\alpha,\beta=1,2,\cdots,6,~\alpha\neq\beta.
\end{equation}
Moreover,
\begin{equation}\label{a11_det}
\frac{1}{C}I\leq\,D\leq\,CI,
\end{equation}
where $I$ is the $3\times3$ identity matrix, and
\begin{equation}\label{a11_det}
\frac{|\ln\epsilon|^{3}}{C}\leq\det{a_{11}}\leq\,C\,|\ln\epsilon|^{3}.
\end{equation}
\end{lemma}

\begin{remark}
Roughly speaking, the estimates of $A$ and $B$ in Lemma \ref{lem_a_11} is that, for some positive constants $c_{1},c_{2},c_{3}$, independent of $\epsilon$,
$$A\sim
\begin{pmatrix}
  ~c_{1}|\ln\epsilon| & O(1) & O(1) ~\\\\
  O(1) & c_{2}|\ln\epsilon| & O(1) \\\\
 O(1) & O(1) & c_{3}|\ln\epsilon|~ \\
\end{pmatrix},
\quad\mbox{and}\quad\,B=
\begin{pmatrix}
  ~O(1) & O(1) & O(1) \\\\
  O(1) & O(1) & O(1) \\\\
  O(1) & O(1) & O(1)~ \\
\end{pmatrix}.
$$
\end{remark}

We postpone the proof of Lemma \ref{lem_a_11} and first make use of it to prove \eqref{C1-C2}.

\begin{proof}[Proof of \eqref{C1-C2}]

For convenience, we introduce notations
$$X_{1}:=\Big(C_{1}^{1}-C_{2}^{1},C_{1}^{2}-C_{2}^{2},C_{1}^{3}-C_{2}^{3}\Big)^{T},\quad
X_{2}:=\Big(C_{1}^{4}-C_{2}^{4},C_{1}^{5}-C_{2}^{5},C_{1}^{6}-C_{2}^{6}\Big)^{T},$$
and
$$P_{1}:=(p_{1},p_{2},p_{3})^{T},\quad\,P_{2}:=(p_{4},p_{5},p_{6})^{T}.$$
Now \eqref{C1C2_4} can be rewritten as
\begin{equation}\label{axp}
\begin{pmatrix}
  A & B ~\\\\
  C & D~ \\
\end{pmatrix}
\begin{pmatrix}
 ~ X_{1}~ \\\\
  ~X_{2} ~\\
\end{pmatrix}
=\begin{pmatrix}
 ~ P_{1}~ \\\\
 ~ P_{2} ~\\
\end{pmatrix}.
\end{equation}
By Lemma \ref{lem_ap} and Lemma \ref{lem_a_11}, matrices $A,B,D$
satisfy the assumptions of Lemma \ref{lem_matrix} in Appendix with $m=3$, $\gamma=|\ln\epsilon|$ and $\theta=\frac{1}{C}$. Applying Lemma \ref{lem_matrix}, we have
$$\begin{pmatrix}
  A & B ~\\\\
  C & D~ \\
\end{pmatrix}^{-1}
=\begin{pmatrix}
  A^{-1}+O(\frac{1}{|\ln\epsilon|^{2}}) & O(\frac{1}{|\ln\epsilon|}) ~\\\\
  * & *~ \\
\end{pmatrix}.$$
It follows from \eqref{axp} that
$$|X_{1}|=\sqrt{\sum_{\alpha=1}^{3}\left|C_{1}^{\alpha}-C_{2}^{\alpha}\right|}\leq\,\frac{C}{|\ln\epsilon|}.$$
Thus, the proof of Proposition \ref{prop_C1-C2} is completed.
\end{proof}

We are now in position to complete the proof of Proposition \ref{prop_gradient}.

\begin{proof}[Proof of Proposition \ref{prop_gradient}]
Estimates \eqref{mainev0}-\eqref{mainev1+23}  have been proved
in Lemma \ref{lem_v3_v1+v2}. Under assumption \eqref{omega},
$$\frac{1}{C}(\epsilon+|x'|^{2})\leq\epsilon+\mathrm{dist}^{2}(x,\overline{P_{1}P_{2}})\leq\,C
(\epsilon+|x'|^{2}),\quad~x\in\Omega_{R}.$$
Estimate \eqref{mainev1} in $\Omega_{R}$ follows from \eqref{nabla_w_ialpha} and \eqref{nablau_123}. Thus, using \eqref{nablavialpha_outomega1/2}, \eqref{mainev1} is proved.
Combining \eqref{nabla_v13} and
\eqref{nablavi3_outomega1/2} yields estimate \eqref{mainevi3}. Estimate \eqref{maineC}  and
estimate \eqref{maineC1-C2}  has been proved in
Proposition \ref{prop_C1-C2}.
 The proof of Proposition \ref{prop_gradient} is finished.
\end{proof}

\begin{proof}[Proof of Lemma \ref{lem_a_11}]

\noindent{\bf STEP 1.} Proof of \eqref{a11_1122} and \eqref{a11_456}.

For any
$\xi=(\xi_{1},\xi_{2},\cdots,\xi_{6})^{T}\neq0$, by \eqref{coeff5},
\begin{align}\label{positiveofa}
\xi^{T}a_{11}\xi&=\int_{\widetilde{\Omega}}\left(\mathbb{C}^{0}e\left(\xi_{\alpha}v_{1}^{\alpha}\right),
e\left(\xi_{\beta}v_{1}^{\beta}\right)\right)dx
\geq\frac{1}{C}\int_{\widetilde{\Omega}}\left|e\left(\xi_{\alpha}v_{1}^{\alpha}\right)\right|^{2}dx>0.
\end{align}
In the last inequality we have used the fact that $e\left(\xi_{\alpha}v_{1}^{\alpha}\right)$ is not identically zero in $\widetilde{\Omega}$. Indeed, if $e\left(\xi_{\alpha}v_{1}^{\alpha}\right)=0$  in $\widetilde{\Omega}$, then $\sum_{\alpha=1}^{6}\xi_{\alpha}v_{1}^{\alpha}=\sum_{i=1}^{6}a_{i}\psi^{i}$  in $\widetilde{\Omega}$, for some constants $a_{i}$, $i=1,2,\cdots,6$. On the other hand, $\sum_{i=1}^{6}\xi_{\alpha}v_{1}^{\alpha}=0$ on $\partial{D}_{2}$, so by Lemma \ref{lem_xi}, $a_{i}=0$, $\forall~i$. Thus on $\partial{D}_{1}$, $\sum_{i=1}^{6}\xi_{\alpha}v_{1}^{\alpha}=\sum_{i=1}^{6}\xi_{\alpha}\psi^{\alpha}=0$, which implies, again using Lemma \ref{lem_xi}, that $\xi=0$. A contradiction. \eqref{positiveofa} implies that $a_{11}$ is positive definite.

By \eqref{coeff2} and \eqref{mainev1}, we have
$$a_{11}^{\alpha\alpha}=\int_{\widetilde{\Omega}}\left(\mathbb{C}^{0}e\left(v_{1}^{\alpha}\right),e\left(v_{1}^{\alpha}\right)\right)dx
\leq\,C\int_{\widetilde{\Omega}}\left|\nabla{v}_{1}^{\alpha}\right|^{2}dx\leq\,C|\ln\epsilon|,\quad\,\alpha=1,2,3.$$
By \eqref{coeff2} agian, we have
$$a_{11}^{\alpha\alpha}=\int_{\widetilde{\Omega}}\left(\mathbb{C}^{0}e\left(v_{1}^{\alpha}\right),e\left(v_{1}^{\alpha}\right)\right)dx
\geq\frac{1}{C}\int_{\widetilde{\Omega}}\left|e\left(v_{1}^{\alpha}\right)\right|^{2}dx
\geq\frac{1}{C}\int_{\Omega_{R}}\left|\partial_{x_{3}}(v_{1}^{\alpha})_{\alpha}\right|^{2}dx,\quad\alpha=1,2,3.$$
Notice that $(v_{1}^{\alpha})_{\alpha}|_{\partial{D}_{1}}=\bar{u}|_{\partial{D}_{1}}=1$, $(v_{1}^{\alpha})_{\alpha}|_{\partial{D}_{2}}=\bar{u}|_{\partial{D}_{2}}=0$, and recalling the definition of $\bar{u}$, \eqref{ubar}, $\bar{u}(x',x_{3})$ is linear in $x_{3}$ for fixed $x'$, so $\bar{u}(x',\cdot)$ is harmonic, hence its energy is minimal, that is,
$$
\int_{h_{2}(x')-\frac{\epsilon}{2}}^{h_{1}(x')+\frac{\epsilon}{2}}\left|\partial_{x_{3}}(v_{1}^{\alpha})_{\alpha}\right|^{2}dx_{3}
\geq\int_{h_{2}(x')-\frac{\epsilon}{2}}^{h_{1}(x')+\frac{\epsilon}{2}}|\partial_{x_{3}}\bar{u}|^{2}dx_{3}
=\frac{1}{\epsilon+h_{1}(x')-h_{2}(x')}.$$
Integrating on $|x'|<R$, we obtain
\begin{align*}
\int_{\Omega_{R}}\left|\partial_{x_{3}}(v_{1}^{\alpha})_{\alpha}\right|^{2}dx
&=\int_{|x'|<R}\int_{h_{2}(x')-\frac{\epsilon}{2}}^{h_{1}(x')+\frac{\epsilon}{2}}
\left|\partial_{x_{3}}(v_{1}^{\alpha})_{\alpha}\right|^{2}dx_{3}dx_{2}dx_{1}\\
&\geq\frac{1}{C}\int_{|x'|<R}\frac{dx'}{\epsilon+|x'|^{2}}
\geq\frac{|\ln\epsilon|}{C}.
\end{align*}
Thus,
\begin{equation}\label{a11_lowbound}
a_{11}^{\alpha\alpha}\geq\frac{|\ln\epsilon|}{C},\quad\alpha=1,2,3.
\end{equation}
Estimate \eqref{a11_1122} is proved.

By \eqref{lem31.02}, we have
$$a_{11}^{\alpha\alpha}=\int_{\widetilde{\Omega}}\left(\mathbb{C}^{0}e\left(v_{1}^{\alpha}\right),e\left(v_{1}^{\alpha}\right)\right)dx\leq\,C,\quad\alpha=4,5,6.$$
By the same argument, the claim (4.18) in \cite{bll} for higher dimensions still holds. Therefore
$$a_{11}^{\alpha\alpha}\geq\frac{1}{C}\int_{\Omega_{R}\setminus\Omega_{R}}|e\left(v_{1}^{\alpha}\right)|^{2}dx
\geq\frac{1}{C}\int_{\Omega_{R}\setminus\Omega_{R}}|\nabla{v}_{1}^{\alpha}|^{2}dx
\geq\frac{1}{C},\quad\alpha=4,5,6.
$$
Estimate \eqref{a11_456} is proved.

\noindent{\bf STEP 2.} We deal with the cases $\alpha\neq\beta$. Proof of \eqref{a11_12}.

By definition,
$$a_{11}^{\alpha\beta}=a_{11}^{\beta\alpha}=\int_{\widetilde{\Omega}}\left(\mathbb{C}^{0}e\left(v_{1}^{\alpha}\right),e(v_{1}^{\beta})\right)dx
=-\int_{\partial{D}_{1}}\frac{\partial{v}_{1}^{\alpha}}{\partial\nu_{0}}\bigg|_{+}\cdot\psi^{\beta}\,dS.$$
First,
\begin{align*}
a_{11}^{12}&=-\int_{\partial{D}_{1}}\frac{\partial{v}_{1}^{1}}{\partial\nu_{0}}\bigg|_{+}\cdot\psi^{2}\,dS\\
&=-\int_{\partial{D}_{1}}\left(\lambda\Big(\nabla\cdot{v}_{1}^{1}\Big)n_{2}
+\mu\bigg(\Big(\nabla{v}_{1}^{1}+(\nabla{v}_{1}^{1})^{T}\Big)\vec{n}\bigg)_{2}\right)\,dS\\
&=-\int_{\partial{D}_{1}}\left(\lambda\Big(\sum_{i=1}^{3}\partial_{x_{k}}({v}_{1}^{1})_{k}\Big)n_{2}
+\mu\sum_{l=1}^{3}\Big(\partial_{x_{2}}({v}_{1}^{1})_{l}+\partial_{x_{l}}({v}_{1}^{1})_{2}\Big)n_{l}\right)\,dS.
\end{align*}
We only need to estimate the integral on the part $\partial{D}_{1}\cap{B}_{R}$, because the rest is bounded. On boundary $\partial{D}_{1}\cap{B}_{R}$, we have $$\vec{n}=\frac{1}{\sqrt{1+|\nabla_{x'}h_{1}|^{2}}}\Big(-\partial_{x_{1}}h_{1},-\partial_{x_{2}}h_{1},1\Big).$$
Clearly, using \eqref{h1h20}-\eqref{h1h2},
$$|n_{1}|\leq\,C|x'|,\quad|n_{2}|\leq\,C|x'|,\quad\,n_{3}=\frac{1}{\sqrt{1+|\nabla_{x'}h_{1}|^{2}}}.$$
Combining with the estimates \eqref{nablav1}, we have
$$\int_{\partial{D}_{1}\cap{B}_{R}}\Big|~\Big(\sum_{i=1}^{3}\partial_{x_{k}}({v}_{1}^{1})_{k}\Big)n_{2}~\Big|\,dS\leq
\int_{\partial{D}_{1}\cap{B}_{R}}\frac{C|x'|}{\epsilon+|x'|^{2}}\,dS\leq\,C.$$
Using the definition of $\bar{u}_{1}^{1}$, estimates \eqref{nabla_w_ialpha}, \eqref{nablav1}, \eqref{nablax'v1}, and
$$\Big|\int_{\partial{D}_{1}\cap{B}_{R}}\partial_{x_{3}}({v}_{1}^{1})_{2}n_{3}\,dS\Big|\leq
\Big|\int_{\partial{D}_{1}\cap{B}_{R}}\partial_{x_{3}}(\bar{u}_{1}^{1})_{2}n_{3}\,dS\Big|
+\Big|\int_{\partial{D}_{1}\cap{B}_{R}}\partial_{x_{3}}({w}_{1}^{1})_{2}n_{3}\,dS\Big|\leq\,C,$$
we obtain
$$\int_{\partial{D}_{1}\cap{B}_{R}}\Big|~\sum_{l=1}^{3}\partial_{x_{2}}({v}_{1}^{1})_{l}n_{l}
+\sum_{l=1}^{3}\partial_{x_{l}}({v}_{1}^{1})_{2}n_{l}~\Big|\,dS\leq\,C.$$
Therefore
$$\left|a_{11}^{12}\right|=\left|a_{11}^{21}\right|\leq\,C.$$

By the same way
\begin{align*}
a_{11}^{13}&=-\int_{\partial{D}_{1}}\left(\lambda\Big(\nabla\cdot{v}_{1}^{1}\Big)n_{3}
+\mu\bigg(\Big(\nabla{v}_{1}^{1}+(\nabla{v}_{1}^{1})^{T}\Big)\vec{n}\bigg)_{3}\right)\,dS\\
&=-\int_{\partial{D}_{1}}\left(\lambda\Big(\sum_{i=1}^{3}\partial_{x_{k}}({v}_{1}^{1})_{k}\Big)n_{3}
+\mu\sum_{l=1}^{3}\Big(\partial_{x_{3}}({v}_{1}^{1})_{l}+\partial_{x_{l}}({v}_{1}^{1})_{3}\Big)n_{l}\right)\,dS.
\end{align*}
For the terms $\partial_{x_{k}}({v}_{1}^{1})_{l}$, $k=1,2,l=1,2,3$, use the estimates \eqref{nablax'v1}, for $k=l=3$, use the definition $\bar{u}_{1}^{1}$ and the estimates \eqref{nabla_w_ialpha} to obtain
$$\Big|\int_{\partial{D}_{1}}\partial_{x_{3}}({v}_{1}^{1})_{3}n_{3}\,dS\Big|=
\Big|\int_{\partial{D}_{1}}\partial_{x_{3}}(\bar{u}_{1}^{1})_{3}n_{3}\,dS\Big|
+\Big|\int_{\partial{D}_{1}}\partial_{x_{3}}({w}_{1}^{1})_{3}n_{3}\,dS\Big|\leq\,C.$$
Therefore
$$\left|a_{11}^{13}\right|=\left|a_{11}^{31}\right|\leq\,C.$$

By the definition and the same reason,
\begin{align*}
a_{11}^{14}=&-\int_{\partial{D}_{1}}\left(\lambda\Big(\nabla\cdot{v}_{1}^{1}\Big)\vec{n}
+\mu\Big(\nabla{v}_{1}^{1}+(\nabla{v}_{1}^{1})^{T}\Big)\vec{n}\right)\cdot
\begin{pmatrix}
x_{2} \\
-x_{1} \\
0 \\
\end{pmatrix}\,dS
\\
=&-\int_{\partial{D}_{1}}\left(\lambda\Big(\sum_{i=1}^{3}\partial_{x_{k}}({v}_{1}^{1})_{k}\Big)n_{1}
+\mu\sum_{l=1}^{3}\Big(\partial_{x_{1}}({v}_{1}^{1})_{l}+\partial_{x_{l}}({v}_{1}^{1})_{1}\Big)n_{l}\right)x_{2}\,dS\\
&+\int_{\partial{D}_{1}}\left(\lambda\Big(\sum_{i=1}^{3}\partial_{x_{k}}({v}_{1}^{1})_{k}\Big)n_{2}
+\mu\sum_{l=1}^{3}\Big(\partial_{x_{2}}({v}_{1}^{1})_{l}+\partial_{x_{l}}({v}_{1}^{1})_{2}\Big)n_{l}\right)x_{1}\,dS
\end{align*}
is bounded. $a_{11}^{15}=a_{11}^{51}$ and $a_{11}^{16}=a_{11}^{61}$ are also bounded, essentially the same as above.

\begin{align*}
a_{11}^{23}&=-\int_{\partial{D}_{1}}\left(\lambda\Big(\nabla\cdot{v}_{1}^{2}\Big)n_{3}
+\mu\bigg(\Big(\nabla{v}_{1}^{1}+(\nabla{v}_{1}^{2})^{T}\Big)\vec{n}\bigg)_{3}\right)\,dS\\
&=-\int_{\partial{D}_{1}}\left(\lambda\Big(\sum_{i=1}^{3}\partial_{x_{k}}({v}_{1}^{2})_{k}\Big)n_{3}
+\mu\sum_{l=1}^{3}\Big(\partial_{x_{3}}({v}_{1}^{2})_{l}+\partial_{x_{l}}({v}_{1}^{2})_{3}\Big)n_{l}\right)\,dS
\end{align*}
is the same as $a_{11}^{12}$. While $a_{11}^{24}=a_{11}^{42}$ and $a_{11}^{34}=a_{11}^{43}$ are the same as $a_{11}^{14}$. $a_{11}^{25}=a_{11}^{52}$ and $a_{11}^{26}=a_{11}^{62}$,  $a_{11}^{35}=a_{11}^{53}$ and $a_{11}^{36}=a_{11}^{63}$ are all the same.

\begin{align*}
a_{11}^{45}=&-\int_{\partial{D}_{1}}\left(\lambda\Big(\nabla\cdot{v}_{1}^{4}\Big)\vec{n}
+\mu\Big(\nabla{v}_{1}^{1}+(\nabla{v}_{1}^{1})^{T}\Big)\vec{n}\right)\cdot
\begin{pmatrix}
x_{3} \\
0 \\
-x_{1} \\
\end{pmatrix}\,dS
\\
=&-\int_{\partial{D}_{1}}\left(\lambda\Big(\sum_{i=1}^{3}\partial_{x_{k}}({v}_{1}^{4})_{k}\Big)n_{1}
+\mu\sum_{l=1}^{3}\Big(\partial_{x_{1}}({v}_{1}^{4})_{l}+\partial_{x_{l}}({v}_{1}^{4})_{1}\Big)n_{l}\right)x_{3}\,dS\\
&+\int_{\partial{D}_{1}}\left(\lambda\Big(\sum_{i=1}^{3}\partial_{x_{k}}({v}_{1}^{4})_{k}\Big)n_{3}
+\mu\sum_{l=1}^{3}\Big(\partial_{x_{2}}({v}_{1}^{4})_{l}+\partial_{x_{l}}({v}_{1}^{4})_{3}\Big)n_{l}\right)x_{1}\,dS
\end{align*}
is much better. $a_{11}^{56}=a_{11}^{65}$ is the same better. Estimate \eqref{a11_12} is proved.

\noindent{\bf STEP 3.} We will show $$D\geq\frac{1}{C}I$$ for some constant $C$, independent of $\epsilon$.

For $\xi\in\mathbb{R}^{3}$, $|\xi|=1$, using \eqref{coeff4_strongelyconvex}, we have
$$\sum_{\alpha,\beta=4,5,6}a_{11}^{\alpha\beta}\xi_{\alpha}\xi_{\beta}
=\int_{\widetilde{\Omega}}\Big(\mathbb{C}^{0}e\Big(\sum_{\alpha=4}^{6}\xi_{\alpha}v_{1}^{\alpha}\Big),
e\Big(\sum_{\beta=4}^{6}\xi_{\beta}v_{1}^{\beta}\Big)\Big)dx
\geq\frac{1}{C}\int_{\widetilde{\Omega}}\Big|e\Big(\sum_{\alpha=4}^{6}\xi_{\alpha}v_{1}^{\alpha}\Big)\Big|^{2}dx.$$
We claim that there exists a constant $C>0$, independent of $\epsilon$, such that
\begin{equation}
\int_{\widetilde{\Omega}}\Big|e\Big(\sum_{\alpha=4}^{6}\xi_{\alpha}v_{1}^{\alpha}\Big)\Big|^{2}dx\geq\frac{1}{C},
\quad\forall~\xi\in\mathbb{R}^{3},~|\xi|=1.
\end{equation}
Indeed, if not, then there exists $\epsilon_{i}\rightarrow0^{+}$, $|\xi^{i}|=1$, such that
\begin{equation}\label{equ_ifnot}
\int_{\widetilde{\Omega}_{\epsilon_{i}}}\Big|e\Big(\sum_{\alpha=4}^{6}\xi_{\alpha}^{i}v_{1}^{\alpha}\Big)\Big|^{2}dx\rightarrow0.
\end{equation}
Here and in the following proof we use the notations $D_{1}^{*}:=(0,0,-\frac{\epsilon}{2})+D_{1}$, $D_{2}^{*}:=(0,0,\frac{\epsilon}{2})+D_{2}$, $\widetilde{\Omega}^{*}:=\Omega\setminus\overline{D_{1}^{*}\cup{D}_{2}^{*}}$, and $\widetilde{\Omega}_{\epsilon}=\Omega\setminus\overline{D_{1}\cup{D}_{2}}$. The corresponding solution of \eqref{v1alpha} with $\alpha=4,5,6$ is denoted as $v_{1}^{\alpha}(\epsilon)$.
Since $v_{1}^{\alpha}(\epsilon_{i})=0$ on $\partial{D}_{2}$, it follows from the second Korn's inequality (see theorem 2.5 in \cite{osy}) that there exists a constant $C$, independent of $\epsilon$, such that
$$\|v_{1}^{\alpha}(\epsilon_{i})\|_{H^{1}(\widetilde{\Omega}_{\epsilon_{i}}\setminus{B}_{\bar{r}})}\leq\,C,$$
where $\bar{r}>0$ is fixed. Then there exists a subsequence, we still denote $\{v_{1}^{\alpha}(\epsilon_{i})\}$, such that
$$v_{1}^{\alpha}(\epsilon_{i})\rightharpoonup\bar{v}_{1}^{\alpha},
\quad\mbox{in}~H^{1}(\widetilde{\Omega}_{\epsilon_{i}}\setminus{B}_{\bar{r}}),\quad\mbox{as}~\epsilon_{i}\rightarrow0.$$
By the assumption \eqref{equ_ifnot}, there exists $\bar{\xi}$ such that
$$\xi^{i}\rightarrow\bar{\xi},\quad\mbox{as}~\epsilon_{i}\rightarrow0,\quad\mbox{with}~~|\bar{\xi}|=1,$$
and
$$\int_{\widetilde{\Omega}^{*}}\Big|e\Big(\sum_{\alpha=4}^{6}\bar{\xi}_{\alpha}\bar{v}_{1}^{\alpha}\Big)\Big|^{2}=0.$$
This implies that
$$e\Big(\sum_{\alpha=4}^{6}\bar{\xi}_{\alpha}\bar{v}_{1}^{\alpha}\Big)=0,\quad\mbox{in}~\widetilde{\Omega}^{*}.$$
Hence, for some constants $\{b_{\beta}\}$,
$\sum_{\alpha=4}^{6}\bar{\xi}_{\alpha}\bar{v}_{1}^{\alpha}=\sum_{\beta=1}^{6}b_{\beta}\psi^{\beta}$ in $\widetilde{\Omega}^{*}$.
Since
$\sum_{\beta=1}^{6}b_{\beta}\psi^{\beta}=0$, on $\partial{D}_{2}^{*}$,
it follows from Lemma \ref{lem_xi} that $b_{\beta}=0$, $\forall~\beta$. Thus,
$\sum_{\alpha=4}^{6}\bar{\xi}_{\alpha}\bar{v}_{1}^{\alpha}=0$ in $\widetilde{\Omega}^{*}$.
Restricted on $\partial{D}_{1}^{*}$, it says that
$\sum_{\alpha=4}^{6}\bar{\xi}_{\alpha}\psi^{\alpha}=0$ on $\partial{D}_{1}^{*}$.
This yields, using again Lemma \ref{lem_xi}, $\bar{\xi}_{\alpha}=0$, $\alpha=4,5,6$, which contradicts with $|\bar{\xi}|=1$.

\eqref{a11_det} is immediately proved by using \eqref{a11_1122} and \eqref{a11_12}. The proof of Lemma \ref{lem_a_11} is finished.
\end{proof}

\section{The proof of Theorem \ref{mainthm2}}\label{sec_higherdimensions}

Define $v_{i}^{\alpha}$ and $v_{0}$ by \eqref{v1alpha} and \eqref{v3}. By a decomposition similar to \eqref{decom_u},
\begin{equation}\label{nablau_dec4}
\nabla{u}=\sum_{i=1}^{2}\sum_{\alpha=1}^{\frac{d(d+1)}{2}}C_{i}^{\alpha}\nabla{v}_{i}^{\alpha}+\nabla{v}_{0},\quad\mbox{in}~\widetilde{\Omega}.
\end{equation}
It follows that
\begin{equation}\label{nablau4}
\left|\nabla{u}\right|
\leq\sum_{i=1}^{2}\sum_{\alpha=1}^{\frac{d(d+1)}{2}}\left|C_{i}^{\alpha}\right|\left|\nabla{v}_{i}^{\alpha}\right|+\big|\nabla{v}_{0}\big|,\quad\mbox{in}~\widetilde{\Omega}.
\end{equation}

As in Section \ref{sec_gradient123}, we write $x=(x',x_{d})$, and let $P_{1},P_{2}, R$ be the same as in Section \ref{sec_gradient123}, and, instead of \eqref{x3},
$$x_{d}=\frac{\epsilon}{2}+h_{1}(x'),\quad\mbox{and}\quad\,x_{d}=-\frac{\epsilon}{2}+h_{2}(x'), \quad\mbox{for}~~|x'|<2R.
$$
$\widehat{\Omega}_{s}(z')$ and $\Omega_{s}=\Omega_{s}(0')$ are defined accordingly. $\bar{u}$, $\underline{u}$ and $\bar{u}_{i}^{\alpha}$ are defined as in \eqref{ubar}, \eqref{u_underline}, \eqref{def:ubar1112} and \eqref{def:ubar2122}, with $x_{3}$ replaced by $x_{d}$, and $\alpha=1,2,\cdots,\frac{d(d+1)}{2}$. We still have \eqref{nablau_bar_outside} and \eqref{nablau_underline_outside}.

\begin{prop}\label{prop5}
Assume the above, let $v_{i}^{\alpha}\in{H}^1(\widetilde{\Omega};\mathbb{R}^{d})$ be the
weak solution of \eqref{v1alpha} with $\alpha=1,2,\cdots,\frac{d(d+1)}{2}$. Then for $i=1,2,~\alpha=1,2,\cdots,\frac{d(d+1)}{2}$,
\begin{equation}\label{energy_w4d}
\int_{\widetilde{\Omega}}\left|\nabla(v_{i}^{\alpha}-\bar{u}_{i}^{\alpha})\right|^{2}dx\leq\,C;
\end{equation}
and 
\begin{equation}\label{nablavialpha_outomega1/24d}
\left\|\nabla{v}_{i}^{\alpha}\right\|_{L^{\infty}(\widetilde{\Omega}\setminus\Omega_{R})}\leq\,C,
\end{equation}

\begin{equation}\label{nabla_w_ialpha4}
\left|\nabla(v_{i}^{\alpha}-\bar{u}_{i}^{\alpha})(x)\right|
\leq\frac{C}{\epsilon+|x'|^{2}},\qquad\forall~~x\in\Omega_{R}.
\end{equation}
Consequently,
\begin{equation}\label{nablav14}
\left|\nabla{v}_{i}^{\alpha}(x)\right|\leq\frac{C}{\epsilon+|x'|^{2}},\qquad\forall~~x\in\Omega_{R}.
\end{equation}
\end{prop}

\begin{proof} The proof is similar to that of Proposition \ref{prop1}, and we only point out the main difference. The proof of \eqref{energy_w4d} and \eqref{nablavialpha_outomega1/24d} are the same as that of \eqref{energy_w} and \eqref{nablavialpha_outomega1/2}. We prove \eqref{nabla_w_ialpha4}.

(i) For $\alpha=1,2,\cdots,d$, the same as \eqref{nablau_123},
\begin{equation}\label{nablau_123d4}
|\nabla\bar{u}_{i}^{\alpha}(x)|\leq\frac{C}{\epsilon+|x'|^{2}},\qquad\,x\in\Omega_{R},
\end{equation}
and, instead of \eqref{L_ubar_ialpha}, 
\begin{equation}\label{L_ubar_ialpha4}
\left|\mathcal{L}_{\lambda,\mu}\bar{u}_{i}^{\alpha}(x)\right|\leq\,C\sum_{k+l<2d}\left|\partial_{x_{k}x_{l}}\bar{u}(x)\right|\leq\frac{C}{\epsilon+|x'|^{2}}
+\frac{C|x'|}{(\epsilon+|x'|^{2})^{2}},\quad~x\in\Omega_{R}.
\end{equation}
Using \eqref{L_ubar_ialpha4}, we have, instead of \eqref{integal_Lubar11} and \eqref{tildeF111}, for $\sqrt{\epsilon}<|z'|<R$, $0<s<\frac{2|z'|}{3}$,
\begin{align}\label{integal_Lubar114}
\int_{\widehat{\Omega}_{s}(z')}\left|\mathcal{L}_{\lambda,\mu}\bar{u}_{i}^{\alpha}\right|^{2}dx
&\leq\,C\int_{|x'-z'|<s}\left(\frac{1}{\epsilon+|x'|^{2}}+\frac{|x'|^{2}}{(\epsilon+|x'|^{2})^{3}}\right)dx'
\leq\frac{Cs^{d-1}}{|z'|^{4}},
\end{align}
and denoting $\widehat{F}(t):=\int_{\widehat{\Omega}_{t}(z')}|\nabla(v_{i}^{\alpha}-\bar{u}_{i}^{\alpha})|^{2}dx$,
\begin{equation}\label{tildeF111d4}
\widehat{F}(t)\leq\,\left(\frac{C_{0}|z'|^{2}}{s-t}\right)^{2}\widehat{F}(s)+C(s-t)^{2}\frac{s^{d-1}}{|z'|^{4}},
\qquad\forall~0<t<s<\frac{2|z'|}{3}.
\end{equation}
Similar as Case 1 of Step 2 in the proof of Proposition \ref{prop1}, set $t_{i}=\delta+2C_{0}i\,|z'|^{2}$, $i=0,1,2,\cdots$, and let $k=\left[\frac{1}{4C_{0}|z'|}\right]$. Using \eqref{tildeF111d4} with $s=t_{i+1}$ and $t=t_{i}$, we have
$$\widehat{F}(t_{i})\leq\,\frac{1}{4}\widehat{F}(t_{i+1})+\frac{C(t_{i+1}-t_{i})^{2}t_{i+1}^{d-1}}{|z'|^{4}}
\leq\,\frac{1}{4}\widehat{F}(t_{i+1})+C(i+1)^{2}|z'|^{2(d-1)},\quad\,i=1,2,\cdots,k.
$$
After $k$ iterations,  we obtain
$$\int_{\widehat{\Omega}_{\delta}(z')}|\nabla(v_{i}^{\alpha}-\bar{u}_{i}^{\alpha})|^{2}dx=\widehat{F}(t_{0})\leq\,C|z'|^{2(d-1)},\quad\forall~\sqrt{\epsilon}<|z'|<R.$$

Instead of \eqref{integal_Lubar11_in} and \eqref{tildeF111_in}, using \eqref{L_ubar_ialpha4}, for $0\leq|z'|<\sqrt{\epsilon}$, $0<s<\sqrt{\epsilon}$,
\begin{align}\label{integal_Lubar11_ind}
\int_{\widehat{\Omega}_{s}(z')}\left|\mathcal{L}_{\lambda,\mu}\tilde{u}\right|^{2}
&\leq\int_{|x'-z'|<s}\left(\frac{C}{\epsilon+|x'|^{2}}+\frac{C|x'|^{2}}{(\epsilon+|x'|^{2})^{3}}\right)dx'
\leq\,\frac{Cs^{d-1}}{\epsilon^{2}},
\end{align}
and
\begin{equation}\label{tildeF111_ind4}
\widehat{F}(t)\leq\,\left(\frac{C_{0}\epsilon}{s-t}\right)^{2}\widehat{F}(s)
+C(s-t)^{2}\frac{s^{d-1}}{\epsilon^{2}},
\quad\forall~0<t<s<\sqrt{\epsilon}.
\end{equation}
Let $t_{i}=\delta+2C_{0}i\epsilon$, $i=0,1,2,\cdots$ and $k=\left[\frac{1}{4C_{0}\sqrt{\epsilon}}\right]$. By \eqref{tildeF111_ind4} with $s=t_{i+1}$ and $t=t_{i}$, we have
\begin{align*}
\widehat{F}(t_{i})&\leq\,\frac{1}{4}\widehat{F}(t_{i+1})+\frac{C\epsilon^{2}t_{i+1}^{d-1}}{\epsilon^{2}}
\leq\,\frac{1}{4}\widehat{F}(t_{i+1})+C(i+1)^{2}\epsilon^{d-1},\quad\,i=1,2,\cdots,k.
\end{align*}
After $k$ iterations, we have
$$\int_{\widehat{\Omega}_{\delta}(z')}|\nabla(v_{i}^{\alpha}-\bar{u}_{i}^{\alpha})|^{2}dx =\widehat{F}(t_{0})\leq\,C\epsilon^{d-1},\quad\forall~0\leq|z'|<\sqrt{\epsilon}.$$
Therefore, we have, instead of \eqref{energy_w_inomega_z1},
\begin{align}\label{energy_w_inomega_z14}
\int_{\widehat{\Omega}_{\delta}(z')}\left|\nabla{w}_{i}^{\alpha}\right|^{2}dx&\leq
\begin{cases}
\displaystyle C\epsilon^{d-1},&0\leq|z'|\leq\sqrt{\epsilon},\\
\displaystyle C|z'|^{2(d-1)},&\sqrt{\epsilon}<|z'|\leq\,R.
\end{cases}
\end{align}

As in Step 3 of the proof of Proposition \ref{prop1}, we have, instead of \eqref{nablaw11}, 
\begin{equation}\label{nablaw11d4}
\left\|\nabla(v_{i}^{\alpha}-\bar{u}_{i}^{\alpha})\right\|_{L^{\infty}(\widehat{\Omega}_{\delta/2}(z'))}\leq\,
\frac{C}{\delta}\left(\delta^{1-\frac{d}{2}}\left\|\nabla(v_{i}^{\alpha}-\bar{u}_{i}^{\alpha})\right\|_{L^{2}(\widehat{\Omega}_{\delta}(z'))}
+\delta^{2}\left\|\mathcal{L}_{\lambda,\mu}\bar{u}_{i}^{\alpha}\right\|_{L^{\infty}(\widehat{\Omega}_{\delta}(z'))}\right).
\end{equation}
Using \eqref{energy_w_inomega_z14} and \eqref{L_ubar_ialpha4}, we obtain 
\begin{equation}\label{nabla_w_ialpha41}
\left|\nabla(v_{i}^{\alpha}-\bar{u}_{i}^{\alpha})(x)\right|\leq
\begin{cases}
 \frac{C}{\sqrt{\epsilon}},&|x'|\leq\sqrt{\epsilon},\\
 \frac{C}{|x'|},&\sqrt{\epsilon}<|x'|\leq\,R.
\end{cases}
\end{equation}
Consequently, 
\eqref{nablav14} follows from \eqref{nablau_123d4} immediately.

(ii) For $d\leq\alpha\leq\frac{d(d+1)}{2}$, we have
\begin{equation}\label{nablau_123d42}
|\nabla\bar{u}_{i}^{\alpha}(x)|\leq\frac{C|x'|}{\epsilon+|x'|^{2}}+C,\qquad\,x\in\Omega_{R},
\end{equation}
and, instead of \eqref{L_ubar_i3}, 
\begin{equation}\label{L_ubar_i34}
|\mathcal{L}_{\lambda,\mu}\bar{u}_{i}^{\alpha}|
\leq\,C\left(|\nabla\bar{u}|+(\epsilon+|x'|)\sum_{k+l<2d}|\partial_{x_{k}x_{l}}\bar{u}|\right)
\leq\frac{C}{\epsilon+|x'|^{2}},\quad\,x\in\Omega_{R}.
\end{equation}
Using \eqref{L_ubar_i34}, we obtain, for $\sqrt{\epsilon}\leq|z'|\leq\,R$, $0<t<s<\frac{2|z'|}{3}$, instead of \eqref{Lu13},
\begin{align}\label{Lu13d}
\int_{\widehat{\Omega}_{s}(z')}\left|\mathcal{L}_{\lambda,\mu}\tilde{u}\right|^{2}dx
\leq\int_{|x'-z'|<s}\frac{C}{\epsilon+|x'|^{2}}\,dx'\leq\frac{Cs^{d-1}}{|z'|^{2}},
\end{align}
Thus, we have
\begin{equation}\label{tildeF111d42}
\widehat{F}(t)\leq\,\left(\frac{C_{0}|z'|^{2}}{s-t}\right)^{2}\widehat{F}(s)+C(s-t)^{2}\frac{s^{d-1}}{|z'|^{2}},
\qquad\forall~0<t<s<\frac{2|z'|}{3}.
\end{equation}
Taking the same iteration procedure as Case 1 of Step 2 in the proof of Proposition \ref{prop1}, set $t_{i}=\delta+2C_{0}i\,|z'|^{2}$, $i=0,1,2,\cdots$, and let $k=\left[\frac{1}{4C_{0}|z'|}\right]$. Using \eqref{tildeF111d42} with $s=t_{i+1}$ and $t=t_{i}$, we have
$$\widehat{F}(t_{i})\leq\,\frac{1}{4}\widehat{F}(t_{i+1})+\frac{C(t_{i+1}-t_{i})^{2}t_{i+1}^{d-1}}{|z'|^{2}}
\leq\,\frac{1}{4}\widehat{F}(t_{i+1})+C(i+1)^{2}|z'|^{2d},\quad\,i=1,2,\cdots,k.
$$
After $k$ iterations,  we obtain
$$\int_{\widehat{\Omega}_{\delta}(z')}|\nabla(v_{i}^{\alpha}-\bar{u}_{i}^{\alpha})|^{2}dx=\widehat{F}(t_{0})\leq\,C|z'|^{2d},\quad\forall~\sqrt{\epsilon}<|z'|<R.$$

For $0\leq|z'|\leq\sqrt{\epsilon}$, $0<t<s<\sqrt{\epsilon}$, using \eqref{L_ubar_i34}, we have, instead of \eqref{integal_Lubar13_in},
\begin{align}\label{integal_Lubar13_ind}
\int_{\widehat{\Omega}_{s}(z')}\left|\mathcal{L}_{\lambda,\mu}\tilde{u}\right|^{2}dx
\leq\int_{|x'-z'|<s}\frac{C}{\epsilon+|x'|^{2}}\,dx'\leq\frac{Cs^{d-1}}{\epsilon},\qquad\,0<s<\sqrt{\epsilon}.
\end{align}
Then similarly as before, we have
\begin{equation*}\label{tildeF111d42}
\widehat{F}(t)\leq\,\left(\frac{C_{0}|z'|^{2}}{s-t}\right)^{2}\widehat{F}(s)+C(s-t)^{2}\frac{s^{d-1}}{\epsilon},
\qquad\forall~0<t<s<\frac{2|z'|}{3}.
\end{equation*}
and iteration formula
$$\widehat{F}(t_{i})\leq\,\frac{1}{4}\widehat{F}(t_{i+1})+\frac{C(t_{i+1}-t_{i})^{2}t_{i+1}^{d-1}}{\epsilon}
\leq\,\frac{1}{4}\widehat{F}(t_{i+1})+C(i+1)^{2}\epsilon^{d},\quad\,i=1,2,\cdots,k.
$$
Thus, we obtain
\begin{equation}\label{energyw13_outd}
\int_{\widehat{\Omega}_{\delta}(z')}|\nabla{w}_{i}^{\alpha}|^{2}dx\leq
\begin{cases}
C|z'|^{2d},&\sqrt{\epsilon}\leq|z'|<R,\\
C\epsilon^{d},&0\leq|z'|<\sqrt{\epsilon}.
\end{cases}
\end{equation}
Therefore, as in the proof of Proposition \ref{prop3}, using \eqref{nablaw11d4}, \eqref{energyw13_outd} and \eqref{L_ubar_ialpha4}, we have, for $i=1,2,~d\leq\alpha\leq\frac{d(d+1)}{2}$,
\begin{equation}\label{nabla_w13d4}
|\nabla(v_{i}^{\alpha}-\bar{u}_{i}^{\alpha})(x',x_{d})|\leq\,C,\qquad\,x\in\Omega_{R}.
\end{equation}
Consequently, using \eqref{nablau_123d42},
\begin{equation}\label{nabla_v13d4}
|\nabla{v}_{i}^{\alpha}(x',x_{d})|\leq\,\frac{C|x'|}{\epsilon+|x'|^{2}}+C,\qquad\,x\in\Omega_{R}.
\end{equation}
The proof of Proposition \ref{prop5} is completed.
\end{proof}

\begin{proof}[Proof of Theorem \ref{mainthm2}]  By the same argument, using Lemma \ref{lem_xi} for $d\geq4$, we still have \eqref{C1C2bound} for dimensions $d\geq4$. Using Proposition \ref{prop5}, Theorem \ref{mainthm2} follows.
\end{proof}

\section{Appendix: Lemmas on $\Psi$ and matrices}\label{sec_appendix}

We first give a lemma on the linear space of rigid displacement $\Psi$.
\begin{lemma}\label{lem_xi}
Let $\xi$ be an element of $\Psi$, defined by \eqref{def_Psid} with $d\geq2$. 
If $\xi$  vanishes at $d$ distinct  points  $\bar{x}^{1},\bar{x}^{2},\cdots,\bar{x}^{d}$, which do not lie on a $(d-2)$-dimensional plane, then $\xi\equiv0$.
\end{lemma}

\begin{proof}
Since $\xi\in\Psi$, it follows that
$$\xi(x)=Ax+b,$$
for some $b\in\mathbb{R}^{d}$ and some $d\times{d}$ skew symmetric matrix $A$. Let
$$\bar{y}^{i}=\bar{x}^{i}-\bar{x}^{d},\quad\,1\leq\,i\leq\,d-1.$$
By the assumption, $\bar{y}^{1},\cdots,\bar{y}^{d-1}$ is linearly independent.
It follows from $\xi(\bar{x}^{i})=0$ that
$$A\bar{y}^{i}=\xi(\bar{x}^{i})-\xi(\bar{x}^{d})=0,\quad\,1\leq\,i\leq\,d-1.$$ 
Therefore $\mathrm{Rank}~ A\leq1$. This, together with $A^{T}+A=0$, implies $A=0$. Recalling that $\xi(\bar{x}^{1})=0$, we have $b=0$. So $\xi\equiv0$.
\end{proof}

Here we prove a linear algebraic lemma used in the proof of Proposition \ref{prop_C1-C2}. We will use notation $\|B\|=\left(\sum_{i,j}|B_{ij}|^{2}\right)^{1/2}$ for a matrix $B$.

\begin{lemma}\label{lem_matrix}
For $m\geq1$, let $A, D$ be $m\times{m}$ invertible matrices and $B$ and $C$ be $m\times{m}$ matrices satisfying, for some $0<\theta<1$ and $\gamma>1$,
\begin{equation}\label{equ_matrix_abd}
\|A^{-1}\|\leq \frac{1}{\theta\gamma}, \qquad \|B\|+\|C\|+\|D^{-1}\|\leq\frac{1}{\theta}.
\end{equation}
Then there exists $\bar{\gamma}=\bar{\gamma}(m)>1$ and $C(m)>1$, such that  if $\gamma\geq\frac{\bar{\gamma}(m)}{\theta^{4}}$,
$$\begin{pmatrix}
  A & B ~\\\\
  C & D ~\\
\end{pmatrix}$$
is invertible. Moreover,
$$\begin{pmatrix}
  E_{11} & E_{12} ~\\\\
  E_{12}^{T} & E_{22} ~\\
\end{pmatrix}:=\begin{pmatrix}
  A & B ~\\\\
  C & D ~\\
\end{pmatrix}^{-1}
-\begin{pmatrix}
  A^{-1} & 0 \\\\
  0 & D^{-1} \\
\end{pmatrix}$$
satisfies
$$\left\|E_{11}\right\|\leq\frac{C(m)}{\theta^{5}\gamma^{2}},\qquad
\left\|E_{12}\right\|\leq\frac{C(m)}{\theta^{3}\gamma},
\quad\mbox{and}~~\left\|E_{22}\right\|\leq\frac{C(m)}{\theta^{5}\gamma}.$$
\end{lemma}

\begin{proof}
Clearly
$$
\begin{pmatrix}
  I & 0 ~\\\\
  -CA^{-1} & I ~\\
\end{pmatrix}
\begin{pmatrix}
  A & B ~\\\\
  C & D ~\\
\end{pmatrix}=
\begin{pmatrix}
  A & B ~\\\\
  0 & D-CA^{-1}B ~\\
\end{pmatrix},
$$
where $I$ is the $m\times{m}$ identity matrix. Since
$$\left\|CA^{-1}B\right\|\leq\frac{C_{1}(m)}{\theta^{3}\gamma},$$
for some constant $C_{1}(m)$ depending only on $m$, there exists some constant $\gamma_{1}(m)$, depending only on $m$, such that for  $\gamma\geq\frac{\gamma_{1}(m)}{\theta^{4}}$, $D-CA^{-1}B$ is invertible and
\begin{equation}\label{equD2-1}
\left\|\Big(D-CA^{-1}B\Big)^{-1}\right\|\leq\frac{2}{\theta}.
\end{equation}
Then
\begin{align*}
\begin{pmatrix}
  A & B ~\\\\
  C & D ~\\
\end{pmatrix}^{-1}
=&
\begin{pmatrix}
  A & B ~\\\\
  0 & D-CA^{-1}B ~\\
\end{pmatrix}^{-1}
\begin{pmatrix}
  I & 0 ~\\\\
  -CA^{-1} & I ~\\
\end{pmatrix}\\
=&
\begin{pmatrix}
  A^{-1} & -A^{-1}B~\Big(D-CA^{-1}B\Big)^{-1} ~\\\\
  0 & \Big(D-CA^{-1}B\Big)^{-1} ~\\
\end{pmatrix}
\begin{pmatrix}
  I & 0 ~\\\\
  -CA^{-1} & I ~\\
\end{pmatrix}\\
=&
\begin{pmatrix}
  A^{-1}+A^{-1}B~\Big(D-CA^{-1}B\Big)^{-1}CA^{-1} & -A^{-1}B~\Big(D-CA^{-1}B\Big)^{-1} ~\\\\
  -\Big(D-CA^{-1}B\Big)^{-1}CA^{-1} & \Big(D-CA^{-1}B\Big)^{-1} ~\\
\end{pmatrix}.
\end{align*}
The estimates for $|E_{11}|$ and $E_{12}$ follow from \eqref{equ_matrix_abd} and \eqref{equD2-1}. For $|E_{22}|$, we have
$$\|E_{22}\|=\left\|\Big((I-D^{-1}CA^{-1}B)^{-1}-I\Big)\,D^{-1}\right\|
\leq\,C(m)\left\|D^{-1}CA^{-1}B\right\|\left\|D^{-1}\right\|\leq\frac{C(m)}{\theta^{5}\gamma}.$$
The proof is finished.
\end{proof}

\noindent{\bf{\large Acknowledgements.}}
H.G. Li was partially supported by  NSFC (11571042). The work of Y.Y. Li was partially supported by NSF grant
DMS-1501004. All authors were partially supported by NSFC (11371060) and the Fundamental Research Funds for the Central Universities.

\vspace{5mm}


\bibliographystyle{amsplain}
\bibliography{References}

\end{document}